\documentclass[10pt,reqno]{amsart}
\usepackage[top=30truemm,bottom=30truemm,left=25truemm,right=25truemm]{geometry}

\usepackage{amssymb, amsmath, amsthm, amsfonts, latexsym, mathtools}
\usepackage{graphicx}
\usepackage{xcolor}
\usepackage{amscd, bm, enumerate, comment, dsfont}
\usepackage[all]{xy}
\usepackage{hyperref}
\usepackage[nameinlink]{cleveref}
\hypersetup{
setpagesize=false,
bookmarksnumbered=true,%
bookmarksopen=true,%
colorlinks=true,%
linkcolor=orange,
citecolor=blue,
}

\crefname{equation}{}{}

\newtheorem{thm}{Theorem}[section]
\crefname{thm}{Theorem}{Theorems}
\newtheorem{prop}[thm]{Proposition}
\crefname{prop}{Proposition}{Propositions}
\newtheorem{lem}[thm]{Lemma}
\crefname{lem}{Lemma}{Lemmas}
\newtheorem{cor}[thm]{Corollary}
\crefname{cor}{Corollary}{Corollaries}
\newtheorem{conj}[thm]{Conjecture}

\theoremstyle{definition}
\newtheorem*{definition}{Definition}
\newtheorem*{example}{Example}

\theoremstyle{remark}
\newtheorem*{remark}{\textbf{Remark}}

\newcommand{\til}{\widetilde}

\newcommand{\modd}{\operatorname{ mod }}

\newcommand{\ceq}{\coloneq}
\newcommand{\Z}{\mathbb{Z}}

\newcommand{\Span}{\operatorname{Span}}

\newcommand{\R}{\mathbb{R}}
\newcommand{\Q}{\mathbb{Q}}

\newcommand{\wt}{\operatorname{wt}}

\font\sevency=wncyr7
\newcommand{\sha}{\scalebox{0.75}{\hbox{\sevency X}}}
\font\sevency=wncyr7
\newcommand{\Sha}{\hbox{\sevency X}}

\numberwithin{equation}{section}

\makeatletter
\def\@settitle{
  \begin{center}
  \baselineskip18\p@\relax
  \Large\bfseries
  \@title
  \end{center}
}
\makeatother

\title[]{An Explicit Expression for MZVs in Terms of Symmetric MZVs}
\author{Katsumi Kina}
\address{Graduate School of Mathematics, Kyushu University}
\email{kkina@math.kyushu-u.ac.jp}

\begin{document}

\dedicatory{Dedicated to Professor Masanobu Kaneko on the occasion of his retirement from Kyushu University}

\begin{abstract}
We provide another proof of the fact, originally proved by Seidai Yasuda, that symmetric multiple zeta values generate the entire space of multiple zeta values. Furthermore, based on this argument, we present an algorithm for expressing multiple zeta values in terms of symmetric multiple zeta values and products of multiple zeta values. Moreover, we give some results on symmetric and finite multiple zeta values of depth three.
\end{abstract}

\maketitle

{\hypersetup{linkcolor=black}
\tableofcontents
}

\section{Introduction}

Kaneko and Zagier, in \cite{KZ}, introduced the following real numbers, called the $*$-symmetric multiple zeta value $\zeta_{\mathcal{S}}^*(\mathds{k})$ and the $\Sha$-symmetric multiple zeta value $\zeta_{\mathcal{S}}^{\sha}(\mathds{k})$, and presented several properties and interesting conjectures. For integers $k_1,\dots,k_r\geq 1$, we define
\begin{align}
\zeta_{S}^*(k_1,\dots,k_r) &\ceq \sum_{i=0}^{r} (-1)^{k_{i+1}+\cdots+k_r}
\zeta^{*}(k_1,\dots,k_i;T)\zeta^{*}(k_r,\dots,k_{i+1};T),\label{eq:def-stu-sym}
\\
\zeta_{S}^{\sha}(k_1,\dots,k_r) &\ceq \sum_{i=0}^{r} (-1)^{k_{i+1}+\cdots+k_r}
\zeta^{\sha}(k_1,\dots,k_i;T)\zeta^{\sha}(k_r,\dots,k_{i+1};T).\label{eq:def-shu-sym}
\end{align}
Here, the multiple zeta value (MZV) is defined by
$$\zeta(k_1,\dots,k_r) \ceq \sum_{0<n_1<\dots<n_r}\frac{1}{n_1^{k_1}\cdots n_r^{k_r}}\ , \quad (k_1,\dots,k_{r-1}\geq 1,\, k_r\geq 2)$$
and $\zeta^{*}$ (resp. $\zeta^{\sha}$) denotes the stuffle (resp. shuffle) regularized multiple zeta values introduced in \cite{IKZ06}. For an index $\mathds{k}=(k_1,\dots,k_r)$, $r$ is called its depth, and $\wt(\mathds{k})\ceq k_1+\cdots+k_r$ is called its weight.

Let $\mathcal{I} = \bigcup_{r\geq 1}(\Z_{>0})^r$ and $\mathcal{I}_{\mathrm{ad}} = \{(k_1,\dots,k_r)\in \mathcal{I}\mid k_r\geq 2\}$. For $\bullet\in\{*,\Sha\}$, we set
$$\mathcal{Z}_k \ceq \Span_{\Q}\{\zeta(\mathds{k})\mid \mathds{k}\in \mathcal{I}_{\mathrm{ad}},\, \wt(\mathds{k})=k\}
\quad \text{and}\quad
\mathcal{Z}_k^{S,\bullet} \ceq \Span_{\Q}\{\zeta^{\bullet}_{S}(\mathds{k})\mid \mathds{k}\in \mathcal{I},\, \wt(\mathds{k})=k\}.$$
Set $\mathcal{Z}_0=\mathcal{Z}_0^{S,\bullet}=\Q$. Moreover, for $\bullet\in\{*,\Sha\}$ we set
\begin{align*}
\mathcal{Z} \ceq \sum_{k\geq 0}\mathcal{Z}_k\quad \text{and}\quad
\mathcal{Z}^{S,\bullet} \ceq \sum_{k\geq 0}\mathcal{Z}_k^{S,\bullet}.
\end{align*}

For the $*$- and $\Sha$-symmetric multiple zeta values, the following results are well known.

\begin{thm}[Theorem 3 in \cite{KZ}]\label{thm:KZ0}
The equations \cref{eq:def-stu-sym} and \cref{eq:def-shu-sym} are independent of $T$. Furthermore, for any $\mathds{k}\in\mathcal{I}$, we have 
$$\zeta_{S}^{\sha}(\mathds{k}) = \zeta_{S}^{*}(\mathds{k}) \mod \zeta(2)\cdot \mathcal{Z}_{\wt(\mathds{k})-2}.$$
\end{thm}

From \cref{thm:KZ0}, we can define the symmetric multiple zeta value (SMZV) $\zeta_S^{\bullet}(\mathds{k})$ as an element of $\mathcal{Z}/\zeta(2)\mathcal{Z}$, independent of the choice of $\bullet$, either $*$ or $\Sha$. We denote it by $\zeta_S(\mathds{k})$, omitting the symbol $\bullet$. Moreover, we denote by $Z_S(\mathds{k})$ the image of $\zeta^{\sha}(\mathds{k};0)$ in $\mathcal{Z}/\zeta(2)\mathcal{Z}$. It is well known that $Z_S(1)=0$ and $Z_S(2k)=0$ for any $k\geq 1$. For symmetric multiple zeta values, there is an interesting conjecture, as described below.

For $\mathds{k}=(k_1,\dots,k_r)\in\mathcal{I}$, the finite multiple zeta value (FMZV) $\zeta_{\mathcal{A}}(\mathds{k})$ is an element of
$$\mathcal{A} \ceq \prod_{p:\text{prime}}\Z/p\Z \;\Big/ \bigoplus_{p:\text{prime}}\Z/p\Z$$
defined by $\zeta_{\mathcal{A}}(\mathds{k}) \ceq (\zeta_{<p}(\mathds{k})\modd p)_{p:\text{prime}}$, where
$$\zeta_{<N}(\mathds{k}) \ceq \sum_{0<n_1<\dots<n_r<N}\frac{1}{n_1^{k_1}\cdots n_r^{k_r}}.$$
The field $\mathbb{Q}$ is naturally embedded into $\mathcal{A}$, and hence $\mathcal{A}$ can be regarded as a $\mathbb{Q}$-vector space. We denote by $\mathcal{Z}_{\mathcal{A}}$ the subspace of $\mathcal{A}$ generated by all finite multiple zeta values.

\begin{conj}[\cite{KZ}]\label{conj-KZ}
There exists an isomorphism
$$\mathcal{Z}_{\mathcal{A}} \cong \mathcal{Z}/\zeta(2)\mathcal{Z}\ ;\ \zeta_{\mathcal{A}}(\mathds{k}) \mapsto \zeta_S(\mathds{k}).$$
\end{conj}

This conjecture implies that the study of symmetric multiple zeta values is equivalent to that of finite multiple zeta values. In this conjecture, it is nontrivial that $\mathcal{Z}/\zeta(2)\mathcal{Z}$ is generated by symmetric multiple zeta values. This follows from the following result.

\begin{thm}[Theorem 6.1 in \cite{Yas16}]\label{thm:main-thm}
For any integer $k\geq 0$, we have $\mathcal{Z}_k = \mathcal{Z}_k^{S,*}$.
\end{thm}

Furthermore, the following statement can be proved as an application of \cref{thm:main-thm}.

\begin{cor}[see Section 6.1 in \cite{Yas16}]
For any integer $k\geq 0$, we have $\mathcal{Z}_k = \mathcal{Z}_k^{S,\sha}$.
\end{cor}

In this paper, we give another proof of \cref{thm:main-thm}. Yasuda originally proved this theorem in \cite{Yas16} by introducing certain real numbers $\zeta^{\natural,\mathcal{F}}(\mathds{k})$, defined as $\mathbb{Q}$-linear combinations of $\zeta_{\mathcal{S}}^{*}(\mathds{k})$. Here, we instead consider the real numbers $\zeta_2(\mathds{k})$, which appear as coefficients of the multitangent functions $\Psi_{\mathds{k}}(z)$.

It follows immediately from the definitions of $*$-symmetric multiple zeta values and stuffle regularized multiple zeta values that $\mathcal{Z}_k^{S,*}\subset \mathcal{Z}_k$. Thus, it remains to show that every element of $\mathcal{Z}_k$ can be expressed in terms of elements of $\mathcal{Z}_k^{S,*}$. We prove this in \Cref{sec:proof-of-thm}. Based on this proof, in \Cref{sec:algorithm}, we give an algorithm for obtaining such expressions of MZVs in terms of $\Sha$-SMZVs and products of $\Sha$-regularized MZVs. In particular, we carry this out explicitly for MZVs of depths one and two. In \Cref{sec:depth2-rev}, we give another expression for $Z_S(r,s)\in\mathcal{Z}/\zeta(2)\mathcal{Z}$ in terms of SMZVs. In \Cref{sec:depth3}, we study some properties of the spaces of SMZVs and FMZVs of depth three. Finally, \Cref{sec:def-DFMZV} presents the equivalence of two different definitions of $Z_{\mathcal{A}}(r,s)$.

\section*{Acknowledgments}
The author expresses his gratitude to his academic advisor Professor Masanobu Kaneko and Professor Hiroyuki Ochiai. The author would also like to thank Steven Charlton and Minoru Hirose for their valuable advice on this work.

\section{Proof of \cref{thm:main-thm}}\label{sec:proof-of-thm}

\subsection{Preparation for the proof}

Here, we recall several facts needed for the proof of \cref{thm:main-thm}.

\begin{thm}[Proposition 8 and Theorem 5 in \cite{KZ}]\label{thm:KZ}
We have the following.
\begin{enumerate}
\item $\zeta_{S}^{*}(\mathds{k})$ satisfies stuffle product. In particular, $\mathcal{Z}_r^{S,*}\cdot \mathcal{Z}_s^{S,*}\subset \mathcal{Z}_{r+s}^{S,*}$.
\item For any $\mathds{k},\mathds{k}'\in\mathcal{I}$, we have $\zeta_{S}^{\sha}(\mathds{k}\,\Sha\,\mathds{k}') = (-1)^{\wt(\mathds{k}')}\zeta_{S}^{\sha}(\mathds{k},\overleftarrow{\mathds{k}'}))$.
\end{enumerate}
Here, $\overleftarrow{\mathds{k}}$ is the reversal of $\mathds{k}$.
\end{thm}

For integers $k_1,\dots,k_r\geq 1$ with $k_1,k_r\geq 2$, we define the multitangent function by
$$\Psi_{k_1,\dots,k_r}(z) \ceq \sum_{n_1<\dots<n_r}\frac{1}{(z+n_1)^{k_1}\cdots(z+n_r)^{k_r}}.$$
Let $\mathcal{I}_{\mathrm{mt}} \ceq \{(k_1,\dots,k_r)\in \mathcal{I}\mid k_1,k_r\geq 2\}$.

\begin{thm}[Theorem 3 in \cite{Bou14}]
For any index $\mathds{k}\in \mathcal{I}_{\mathrm{mt}}$, there exist real numbers $\zeta_a(\mathds{k})\in \mathcal{Z}_{\wt(\mathds{k})-a}$ such that
$$ \Psi_{\mathds{k}}(z) = \sum_{a\geq 2}\zeta_a(\mathds{k})\Psi_{a}(z), $$
where $\zeta_a(\mathds{k})$ is given by
\begin{equation}\label{eq:zeta_a}
\zeta_a(k_1,\dots,k_r)
=\sum_{i=1}^{r}\sum_{\substack{n_1,\dots,n_{r}\geq 1\\n_1+\cdots+n_r=\wt(\mathds{k})\\n_i=a}}(-1)^{n_1+\cdots+n_i+k_i}\prod_{\substack{j=1\\j\neq i}}^{r}\binom{n_j-1}{k_j-1}\zeta(n_{i-1},\dots,n_1)\zeta(n_{i+1},\dots,n_r).
\end{equation}
\end{thm}

For an index $\mathds{k}=(k_1,\dots,k_r)\in \mathcal{I}_{\mathrm{mt}}$, we denote
$$^{\downarrow}\mathds{k}^{\downarrow} \ceq (k_1-1,k_2,\dots,k_{r-1},k_r-1)
\quad (\text{if }r=1,\text{ then } ^{\downarrow}(k)^{\downarrow} \ceq (k-2)).$$
Furthermore, for an index $\mathds{k}=(k_1,\dots,k_r)\in \mathcal{I}$, Hoffman's dual index $\mathds{k}^{\vee}$ is defined by
$$\mathds{k}^{\vee} = (\underbrace{1,\dots,1}_{k_1}+\underbrace{1,\dots,1}_{k_2}+1,\dots,1+\underbrace{1,\dots,1}_{k_r}).$$

\begin{thm}[Part of a result in \cite{Hir24}]\label{thm:Hirose}
For any index $\mathds{k}\in \mathcal{I}_{\mathrm{mt}}$, we have
$$\zeta_{\mathcal{S}}^{\sha}((^{\downarrow}\mathds{k}^{\downarrow})^{\vee}) = (-1)^{\wt(\mathds{k})}\zeta_2(\mathds{k}).$$
\end{thm}

\subsection{Proof of \cref{thm:main-thm}}

In the following, by relations among multiple zeta values, we mean relations among those of the same weight. Although some parts of the argument using generating functions may appear to ignore the weight, it is in fact distinguished by the total degree of the generating functions. Therefore, in our argument, we do not consider relations between different weights. (It is expected that there are no $\Q$-linear relations among multiple zeta values of different weights.)

To prove \cref{thm:main-thm}, it suffices to show the following.

\begin{thm}\label{prop:existence-mzv}
Any multiple zeta value can be expressed as a $\Q$-linear combination of $\Sha$-symmetric multiple zeta values and products of multiple zeta values.
\end{thm}

Indeed, the case $k=2$ is immediate, since $\mathcal{Z}_2=\mathcal{Z}_2^{S,\ast}$. Now assume, as the induction hypothesis on the weight, that $\mathcal{Z}_r=\mathcal{Z}_r^{S,\ast}$ for all $r<k$. Then, by part (1) of \cref{thm:KZ}, we obtain
$$\mathcal{Z}_{k}^+ \ceq \sum_{\substack{r,s\geq 2 \\ r+s=k}} \mathcal{Z}_{r}\cdot \mathcal{Z}_{s}= \sum_{\substack{r,s\geq 2 \\ r+s=k}} \mathcal{Z}_{r}^{S,*}\cdot \mathcal{Z}_{s}^{S,*}
\subset \mathcal{Z}_{k}^{S,*}.$$
Consequently, by \cref{thm:KZ0} and \cref{prop:existence-mzv}, we have
$$\mathcal{Z}_k^{S,*} = \mathcal{Z}_k^{S,*}+\mathcal{Z}_{k}^+ = \mathcal{Z}_k^{S,\sha}+\mathcal{Z}_{k}^+ = \mathcal{Z}_k.$$

In the following, we prove \cref{prop:existence-mzv}. To prove this, we introduce several generating functions:
\begin{align*}
\mathfrak{Z}(X_1,\dots,X_r) &\ceq \sum_{k_1,\dots,k_{r}\geq 1}\zeta^{\sha}(k_1,\dots,k_r;0)X_1^{k_1-1}\cdots X_r^{k_r-1},
\\
\mathfrak{Z}^{\vee}(X_1,\dots,X_r) &\ceq X_1X_r\sum_{k_1,\dots,k_{r}\geq 1}(-1)^{k_1+\cdots+k_r}\zeta^{\sha}_S((k_1,\dots,k_r)^{\vee})X_1^{k_1-1}\cdots X_r^{k_r-1},
\\
\mathfrak{Z}^S(X_1,\dots,X_r) &\ceq \sum_{i=0}^{r}(-1)^{r-i}\mathfrak{Z}(X_1,\dots,X_i)
\mathfrak{Z}(-X_r,\dots,-X_{i+1}),
\intertext{and}
\mathfrak{A}_n(X_1,\dots,X_r)
&\ceq \sum_{i=1}^{r}(-1)^{i-1} X_i^{n-1} \mathfrak{Z}(X_i-X_{i-1},\dots,X_i-X_1)\mathfrak{Z}(X_{i+1}-X_{i},\dots,X_r-X_i).
\end{align*}
Note that $\zeta^{\bullet}(\mathds{k};0)\in \mathcal{Z}$ for each $\mathds{k}\in\mathcal{I}$ and $\bullet\in\{\ast,\Sha\}$. Then, by a straightforward calculation together with \cref{eq:def-shu-sym} and \cref{eq:zeta_a}, we obtain
\begin{align*}
\mathfrak{Z}^S(X_1,\dots,X_r) &= \sum_{k_1,\dots,k_{r}\geq 1}\zeta^{\sha}_S(k_1,\dots,k_r)X_1^{k_1-1}\cdots X_r^{k_r-1},
\\
\frac{\partial}{\partial X_1}\frac{\partial}{\partial X_r}\mathfrak{A}_2(X_1,\dots,X_r)
&= \sum_{k_1,\dots,k_{r}\geq 1}k_1k_r\zeta_2(k_1+1,k_2,\dots,k_{r-1},k_r+1) X_1^{k_1-1}\cdots X_r^{k_r-1}.
\end{align*}
Hence, by \cref{thm:Hirose}, we have
\begin{equation}\label{eq:A2-ZS}
\frac{\partial}{\partial X_1}\frac{\partial}{\partial X_r}\mathfrak{A}_2(X_1,\dots,X_r)
= \frac{\partial}{\partial X_1}\frac{\partial}{\partial X_r}\mathfrak{Z}^{\vee}(X_1,\dots,X_r).
\end{equation}
On the other hand, we have the shuffle-antipode relation (see Proposition 3.3 in \cite{Bac22}):
\begin{equation*}
\sum_{i=1}^{r}\sum_{\substack{n_1,\dots,n_{r}\geq 1\\n_1+\cdots+n_r=\wt(\mathds{k})\\n_i=1}}(-1)^{n_1+\cdots+n_i+k_i}\prod_{\substack{j=1\\j\neq i}}^{r}\binom{n_j-1}{k_j-1}\zeta^{\sha}(n_{i-1},\dots,n_1)\zeta^{\sha}(n_{i+1},\dots,n_r) = 0.
\end{equation*}
By the above equation and a straightforward calculation, we have
$$\mathfrak{A}_1(X_1,\dots,X_r) = 0.$$
To prove \cref{prop:existence-mzv}, it suffices to prove the following statement.

\begin{thm}\label{thm:to-prove}
For any integer $r\geq 1$, we have $\frac{\partial}{\partial X_r}\mathfrak{Z}(X_1,\dots,X_{r})=0$ under taking modulo $\Sha$-symmetric multiple zeta values and products of multiple zeta values.
\end{thm}

\begin{proof}
In what follows, we write \(\overset{\mathrm{s}}{\equiv}\) and \(\overset{\mathrm{p}}{\equiv}\) to indicate congruence modulo \(\Sha\)-symmetric MZVs and modulo products of MZVs, respectively. After taking modulo, we obtain
\begin{align}
0 \overset{\mathrm{s}}{\equiv} \mathfrak{Z}^S(X_1,\dots,X_{r-1})
&\overset{\mathrm{p}}{\equiv} \mathfrak{Z}(X_1,\dots,X_{r-1})
+(-1)^{r-1}\mathfrak{Z}(-X_{r-1},\dots,-X_1),\label{eq:relation1}
\\
0 = \mathfrak{A}_1(X_1,\dots,X_r)
&\overset{\mathrm{p}}{\equiv} \mathfrak{Z}(X_2-X_1,\dots,X_r-X_1)
+ (-1)^{r-1} \mathfrak{Z}(X_r-X_{r-1},\dots,X_r-X_1),\label{eq:relation2}
\\
\mathfrak{A}_2(X_1,\dots,X_r)
&\overset{\mathrm{p}}{\equiv} X_1 \mathfrak{Z}(X_2-X_1,\dots,X_r-X_1)
+ (-1)^{r-1} X_r \mathfrak{Z}(X_r-X_{r-1},\dots,X_r-X_1).\label{eq:relation3}
\end{align}
By replacing $X_i$ with $X_i-X_r$ for $1\leq i\leq r-1$ in \cref{eq:relation1} and using \cref{eq:relation2}, we obtain
\begin{equation}\label{eq:relation4}
\mathfrak{Z}(X_1-X_r,\dots,X_{r-1}-X_r)
\overset{\mathrm{s,p}}{\equiv} \mathfrak{Z}(X_2-X_1,\dots,X_r-X_1).
\end{equation}
Multiplying \cref{eq:relation2} by $X_r$ and subtracting it from \cref{eq:relation3}, we obtain
\begin{equation}\label{eq:relation5}
\mathfrak{A}_2(X_1,\dots,X_r)
\overset{\mathrm{p}}{\equiv} (X_1-X_r) \mathfrak{Z}(X_2-X_1,\dots,X_r-X_1).
\end{equation}
On the other hand, from \cref{thm:Hirose}, we have
\begin{equation}\label{eq:relation5.5}
\begin{split}
0 \overset{\mathrm{s}}{\equiv} \mathfrak{Z}^{\vee}(X_1,\dots,X_r)
&= \mathfrak{A}_2(X_1,\dots,X_r) - \mathfrak{A}_2(0,X_2,\dots,X_r)
\\
&\quad\quad - \mathfrak{A}_2(X_1,\dots,X_{r-1},0) + \mathfrak{A}_2(0,X_2,\dots,X_{r-1},0).
\end{split}
\end{equation}
Thus, substituting \cref{eq:relation5} into \cref{eq:relation5.5}, we obtain
\begin{align}
\begin{split}\label{eq:relation6}
0 \overset{\mathrm{s,p}}{\equiv} & (X_1-X_r) \mathfrak{Z}(X_2-X_1,\dots,X_r-X_1)
+ X_r \mathfrak{Z}(X_2,\dots,X_r)
\\
&\quad\quad - X_1 \mathfrak{Z}(X_2-X_1,\dots,X_{r-1}-X_1,-X_1).
\end{split}
\end{align}
Then, in \cref{eq:relation6}, replacing $X_i$ with $X_i-X_1$ for $2\leq i\leq r$ and replacing $X_1$ with $-X_1$, we obtain
\begin{equation}\label{eq:relation7}
0 \overset{\mathrm{s,p}}{\equiv} -X_r \mathfrak{Z}(X_2,\dots,X_r)
- (X_1-X_r) \mathfrak{Z}(X_2-X_1,\dots,X_r-X_1)
+ X_1 \mathfrak{Z}(X_2,\dots,X_{r-1},X_1).
\end{equation}
Moreover, by adding \cref{eq:relation6} and \cref{eq:relation7}, we obtain
\begin{equation}\label{eq:relation7.5}
\mathfrak{Z}(X_2-X_1,\dots,X_{r-1}-X_1,-X_1) \overset{\mathrm{s,p}}{\equiv} \mathfrak{Z}(X_2,\dots,X_{r-1},X_1).
\end{equation}
Therefore, applying \cref{eq:relation4} to the left-hand side of \cref{eq:relation7.5}, we obtain
\begin{equation}\label{eq:relation8}
\mathfrak{Z}(X_1,X_2,\dots,X_{r-1}) \overset{\mathrm{s,p}}{\equiv} \mathfrak{Z}(X_2,\dots,X_{r-1},X_1).
\end{equation}
This congruence shows that $\mathfrak{Z}(X_1,\dots,X_{r-1})$ is invariant under cyclic permutations. Combining this with
\begin{equation*}
0 \overset{\mathrm{s}}{\equiv} \frac{\partial}{\partial X_1}\frac{\partial}{\partial X_r}\mathfrak{A}_2(X_1,\dots,X_r)
\overset{\mathrm{p}}{\equiv} \frac{\partial}{\partial X_1}\frac{\partial}{\partial X_r}
(X_1-X_r) \mathfrak{Z}(X_2-X_1,\dots,X_r-X_1),
\end{equation*}
we obtain
\begin{equation*}
0 \overset{\mathrm{s,p}}{\equiv} \frac{\partial}{\partial X_1}\frac{\partial}{\partial X_j}
(X_1-X_j) \mathfrak{Z}(X_2-X_1,\dots,X_r-X_1)
= \sum_{i=2}^{r}\frac{\partial}{\partial X_i}\frac{\partial}{\partial X_j}
(X_j-X_1) \mathfrak{Z}(X_2-X_1,\dots,X_r-X_1)
\end{equation*}
for any $j$ with $2\leq j\leq r$. Combining this with \cref{lem:yasuda-lem} below, we obtain
\begin{equation}\label{eq:relation9}
0 \overset{\mathrm{s,p}}{\equiv} \frac{\partial}{\partial X_j} \mathfrak{Z}(X_2-X_1,\dots,X_r-X_1)
\end{equation}
for any $j$ with $2\leq j\leq r$. Upon setting $X_1=0$, this gives the desired equation.
\end{proof}

\begin{lem}[cf. Lemma 6.6 in \cite{Yas16}]\label{lem:yasuda-lem}
Let $f(X_1,\dots,X_r)\in \R[X_1,\dots,X_r]$ be a homogeneous polynomial of degree $k$. For $j=1,\dots,r$, set
$$F_j\ceq \sum_{i=1}^{r}\frac{\partial}{\partial X_i}\frac{\partial}{\partial X_j}(X_j f)
\quad\text{and}\quad 
F=\sum_{j=1}^{r}F_j.$$
Then we have
\begin{equation}\label{eq:coeff-explicit}
\frac{\partial}{\partial X_j}f
= F_j-\frac{1}{k+r}\frac{\partial}{\partial X_j}(X_jF).
\end{equation}
\end{lem}

\begin{proof}
We have
$$(k+r)\sum_{i=1}^{r}\frac{\partial}{\partial X_i} f
= \sum_{i=1}^{r}\frac{\partial}{\partial X_i}\left(\sum_{j=1}^{r}\frac{\partial}{\partial X_j}(X_j f)\right) = \sum_{j=1}^{r}F_j =F.$$
Therefore, we obtain
\begin{align*}
F_j &= \sum_{i=1}^{r}\frac{\partial}{\partial X_i}\frac{\partial}{\partial X_j}(X_j f)
= X_j\frac{\partial}{\partial X_j}\sum_{i=1}^{r}\frac{\partial}{\partial X_i}f
+ \frac{\partial}{\partial X_j}f
+ \sum_{i=1}^{r}\frac{\partial}{\partial X_i}f
\\
&= \frac{\partial}{\partial X_j}f
+ \frac{1}{k+r}\left(F+X_j\frac{\partial}{\partial X_j}F\right).
\end{align*}
As a consequence, we have
\begin{align*}
\frac{\partial}{\partial X_j}f
= F_j-\frac{1}{k+r}\frac{\partial}{\partial X_j}(X_jF).
\end{align*}
\end{proof}

\begin{remark}
The main difference between the proof of \cref{thm:main-thm} in this paper and Yasuda's proof lies in the type of relations required. In Yasuda's proof, only the extended double shuffle relations (EDSR) are required. However, in the proof given in this paper, we need not only EDSR but also duality relations. In fact, duality relations are used to prove \cref{thm:Hirose}. Since it is conjectured that EDSR provides all relations among multiple zeta values, duality relations are expected to follow from EDSR. However, this inclusion has not yet been proven. Consequently, when considering formal symmetric multiple zeta values (see \cite{BR26}), Yasuda's result has an advantage until such a proof is established.
\end{remark}

\section{An algorithm to obtain an expression of MZVs in terms of SMZVs}\label{sec:algorithm}

By carrying out the argument in the proof of \cref{thm:to-prove} without taking modulo and comparing coefficients of the resulting formal power series, we can explicitly express MZVs in terms of $\Sha$-SMZVs and products of $\Sha$-regularized MZVs (evaluated at $T=0$). To illustrate this concretely, we compute below the cases of MZVs of depths one and two.

Furthermore, by examining the algorithm, we obtain the following properties of this expression. For an MZV of weight $w$ and depth $r$, we have:
\begin{enumerate}
\item The sum of the depths of the two ($\Sha$-regularized) MZVs appearing in each product term is $r$.
\item The depths of the $\Sha$-SMZVs appearing in the expression are either $r$ or $w-r$, where the latter is the depth of the Hoffman dual of an index of weight $w$ and depth $r+1$.
\item The coefficients appearing in the expression lie in $\displaystyle{\frac{1}{w}\Z}$.
\end{enumerate}
Indeed, it follows directly from the proof of \cref{thm:to-prove} that properties (1) and (2) hold. Note also that $\zeta^{\sha}(k_1,\dots,k_r;0)$ can be expressed as a $\Q$-linear combination of MZVs of depth $r$. Moreover, property (3) follows from the fact that the factor $\frac{1}{k+r}$ appearing in \cref{eq:coeff-explicit} is the only factor that can contribute to the denominators of the coefficients.

\begin{example} We have
\begin{align*}
\zeta(2,3) &= \frac{1}{5} \zeta(2) \zeta(3) + \frac{6}{5} \zeta^{\sha}_S(2, 3) + \frac{4}{5} \zeta^{\sha}_S(3, 2) - \frac{4}{5} \zeta^{\sha}_S(1, 2, 2) - \frac{2}{5} \zeta^{\sha}_S(2, 1, 2) - \zeta^{\sha}_S(2, 2, 1).
\\
\zeta(1,2,2) &= \frac{1}{5} \zeta(2) \zeta(1, 2) - \frac{2}{5} \zeta^{\sha}_S(1, 4) - \frac{1}{5} \zeta^{\sha}_S(3, 2)
\\
&\quad\quad+ \frac{8}{5} \zeta^{\sha}_S(1, 1, 3) + \frac{4}{5} \zeta^{\sha}_S(1, 2, 2) + \zeta^{\sha}_S(2, 1, 2) + \frac{1}{5} \zeta^{\sha}_S(2, 2, 1) + \frac{2}{5} \zeta^{\sha}_S(3, 1, 1).
\end{align*}
\end{example}

Moreover, this procedure yields a representation of MZVs as linear combinations of SMZVs in $\mathcal{Z}/\zeta(2)\mathcal{Z}$. We summarize below the algorithm for obtaining such a representation.

\begin{enumerate}
\item For a given $r>0$, carry out the computation leading to \cref{eq:relation9} without taking modulo.
\item Expand both sides of \cref{eq:relation9} as formal power series without taking modulo and compare their coefficients. This yields an explicit expression for MZVs as linear combinations of $\Sha$-SMZVs and products of $\Sha$-regularized MZVs.
\item Rewrite the $\Sha$-regularized MZVs as linear combinations of MZVs. We then obtain an explicit expression for MZVs as linear combinations of $\Sha$-SMZVs and products of MZVs.
\item Rewrite each MZV appearing as a factor in these products as a linear combination of $\Sha$-SMZVs and products of $\Sha$-regularized MZVs.
\item Repeating steps (3) and (4), we obtain an explicit expression for MZVs as linear combinations of products of $\Sha$-SMZVs.
\item Since $\Sha$-SMZVs satisfy the harmonic product relation in $\mathcal{Z}/\zeta(2)\mathcal{Z}$, the expression obtained in step (5) yields an explicit representation of MZVs as linear combinations of SMZVs in $\mathcal{Z}/\zeta(2)\mathcal{Z}$.
\end{enumerate}

Although the resulting computation is quite involved and it is difficult to write the result as a general explicit formula, the algorithm can be implemented on a computer.

\subsection{The depth 1 case} 
Without taking the modulo in \cref{eq:relation5}, and then partially differentiating with respect to $X_1$ and $X_r$, we obtain
$$\frac{\partial}{\partial X_1}\frac{\partial}{\partial X_2}\mathfrak{A}_2(X_1,X_2)
= \frac{\partial}{\partial X_1}\frac{\partial}{\partial X_2}(X_1-X_2) \mathfrak{Z}(X_2-X_1)
= \frac{\partial}{\partial X_2}\frac{\partial}{\partial X_2}(X_2-X_1) \mathfrak{Z}(X_2-X_1).$$
Expanding the above generating function, we have
\begin{align*}
\sum_{k_1,k_2\geq 1}(-1)^{k_1+k_2}k_1k_2
&\zeta^{\sha}_S((k_1,k_2)^{\vee})X_1^{k_1-1}X_2^{k_2-1}
= \sum_{k\geq 2} k(k-1)\zeta(k)(X_2-X_1)^{k-2}
\\
&= \sum_{r\geq 1}\sum_{s\geq 1} (r+s)(r+s-1)\binom{r+s-2}{r-1}\zeta(r+s)(-1)^{r-1}X_1^{r-1}X_2^{s-1}.
\end{align*}
By comparing the coefficients, we have
\begin{equation}\label{eq:OSS}
\zeta^{\sha}_{S}(\{1\}^{k_1-1},2,\{1\}^{k_2-1}) = \zeta^{\sha}_S((k_1,k_2)^{\vee}) 
= (-1)^{k_2-1}\binom{k_1+k_2}{k_1}\zeta(k_1+k_2).
\end{equation}
This formula was already given by Ono--Sakurada--Seki (Theorem 3.8 in \cite{OSS21}). Then, the following are well-known results.

\begin{lem}[{see \cite[Theorem 5.3, Proposition 5.5, Theorem 5.8]{MO21};
see also \cite[Corollaire 1.12]{Jar14}}] \label{lem:basic-prop}
We have the following.
\begin{equation}\label{lem:repeated-index}
\zeta_S(\{k\}^{n})=\zeta_S^{\star}(\{k\}^{n})=0,
\end{equation}
\begin{equation}\label{lem:anti-index}
\zeta_S(k_1,k_2,\dots,k_r)
= (-1)^{k_1+k_2+\cdots+k_r}\zeta_S(k_r,\dots,k_2,k_1),
\end{equation}
\begin{equation}\label{lem:Hoffman-dual}
\zeta_S^{\star}(\mathds{k}) = -\zeta_S^{\star}(\mathds{k}^{\vee})
\end{equation}
\begin{equation}\label{lem:stuffle-antipode}
\sum_{i=0}^{r}(-1)^i\zeta_S(k_1,\dots,k_i)\zeta^{\star}_S(k_r,\dots,k_{i+1}) = 0.
\end{equation}
Here, for $\mathds{k}=(k_1,\dots,k_r)\in \mathcal{I}$, the star symmetric multiple zeta value $\zeta_S^{\star}(\mathds{k})$ is defined by
$$\zeta_S^{\star}(\mathds{k}) \ceq \sum_{\Box\text{ is `,' or }`+\text{'}} \zeta_S(k_1\Box\dots\Box k_r).$$
\end{lem}

Using the above lemma, we have
\begin{align*}
\zeta_{S}(\{1\}^{k_1-1},2,\{1\}^{k_2-1})
&\overset{\cref{lem:repeated-index} \text{ and } \cref{lem:stuffle-antipode}}{=} (-1)^{k_1+k_2}\zeta_{S}^{\star}(\{1\}^{k_2-1},2,\{1\}^{k_1-1})
\\
&\overset{\cref{lem:Hoffman-dual}}{=} (-1)^{k_1+k_2+1}\zeta_{S}^{\star}(k_2,k_1)
\overset{\cref{lem:repeated-index} \text{ and } \cref{lem:anti-index}}{=} (-1)^{k_1+k_2}\zeta_{S}(k_1,k_2).
\end{align*}
Recall that $Z_S(\mathds{k})$ is the image of $\zeta^{\sha}(\mathds{k})$ in $\mathcal{Z}/\zeta(2)\mathcal{Z}$. Therefore, for any integers $k_1,k_2\geq 1$, we obtain
\begin{equation}\label{eq:depth1StoZ}
\zeta_{S}(k_1,k_2) = (-1)^{k_1-1}\binom{k_1+k_2}{k_1}Z_S(k_1+k_2)
= (-1)^{k_2}\binom{k_1+k_2}{k_1}Z_S(k_1+k_2)
\end{equation}
from \cref{eq:OSS}, where the second equality follows from the fact that $Z_S(2k)=0$ for $k\geq 1$. This equation was already given in Example 9.4 of \cite{Kan19}. From \cref{eq:depth1StoZ}, we have $\zeta_S(r,s)=0$ if $r+s$ is even.
Therefore, $\zeta_S^{\star}(k_1,k_2,k_3)=\zeta_S(k_1,k_2,k_3)$ if $k_1+k_2+k_3$ is even.

\subsection{The depth 2 case} Equations \cref{eq:relation1,eq:relation2,eq:relation3} without taking the modulo are expressed as
\begin{align*}
\mathfrak{Z}^S(X_1-X_3,X_2-X_3) &= \mathfrak{Z}(X_1-X_3,X_2-X_3)
- \mathfrak{Z}(X_1-X_3)\mathfrak{Z}(X_3-X_2) + \mathfrak{Z}(X_3-X_2,X_3-X_1),
\\
0 = \mathfrak{A}_1(X_1,X_2,X_3) &= \mathfrak{Z}(X_2-X_1,X_3-X_1)
- \mathfrak{Z}(X_3-X_2)\mathfrak{Z}(X_2-X_1) + \mathfrak{Z}(X_3-X_2,X_3-X_1),
\\
\mathfrak{A}_2(X_1,X_2,X_3) &= X_1\mathfrak{Z}(X_2-X_1,X_3-X_1) \notag
\\
&\quad - X_2\mathfrak{Z}(X_3-X_2)\mathfrak{Z}(X_2-X_1) + X_3\mathfrak{Z}(X_3-X_2,X_3-X_1).
\end{align*}
Hence, \cref{eq:relation4,eq:relation5} without taking the modulo are expressed as
\begin{align}
\mathfrak{Z}(X_1-X_3,X_2-X_3) 
&= \mathfrak{Z}(X_2-X_1,X_3-X_1)
 + \mathfrak{Z}^S(X_1-X_3,X_2-X_3)\notag
 \\
& + \mathfrak{Z}(X_1-X_3)\mathfrak{Z}(X_3-X_2)
 - \mathfrak{Z}(X_3-X_2)\mathfrak{Z}(X_2-X_1) \label{eq:relation3.4}
\\
\mathfrak{A}_2(X_1,X_2,X_3) &= (X_1-X_3)\mathfrak{Z}(X_2-X_1,X_3-X_1)
- (X_2-X_3)\mathfrak{Z}(X_3-X_2)\mathfrak{Z}(X_2-X_1) \label{eq:relation3.5}
\end{align}
thus, \cref{eq:relation6} without taking the modulo is expressed as
\begin{align}\label{eq:ZV-depth2}
\begin{split}
\mathfrak{Z}^{\vee}(X_1,X_2,X_3)
&= (X_1-X_3)\mathfrak{Z}(X_2-X_1,X_3-X_1)
- (X_2-X_3)\mathfrak{Z}(X_3-X_2)\mathfrak{Z}(X_2-X_1)
\\
&+ X_3\mathfrak{Z}(X_2,X_3)
+(X_2-X_3)\mathfrak{Z}(X_3-X_2)\mathfrak{Z}(X_2)
\\
&- X_1\mathfrak{Z}(X_2-X_1,-X_1)
+ X_2\mathfrak{Z}(-X_2)\mathfrak{Z}(X_2-X_1)
- X_2\mathfrak{Z}(-X_2)\mathfrak{Z}(X_2).
\end{split}
\end{align}
Applying the substitutions $X_1\mapsto -X_1$, $X_2\mapsto X_2-X_1$, and $X_3\mapsto X_3-X_1$, we obtain the equation corresponding to \cref{eq:relation7}.
\begin{align}\label{eq:ZV-depth2-2}
\begin{split}
\mathfrak{Z}^{\vee}(-X_1,&X_2-X_1,X_3-X_1)
\\
&= (-X_3)\mathfrak{Z}(X_2,X_3)
- (X_2-X_3)\mathfrak{Z}(X_3-X_2)\mathfrak{Z}(X_2)
\\
&+ (X_3-X_1)\mathfrak{Z}(X_2-X_1,X_3-X_1)
+(X_2-X_3)\mathfrak{Z}(X_3-X_2)\mathfrak{Z}(X_2-X_1)
\\
&+ X_1\mathfrak{Z}(X_2,X_1)
+ (X_2-X_1)\mathfrak{Z}(X_1-X_2)\mathfrak{Z}(X_2)
\\
&- (X_2-X_1)\mathfrak{Z}(X_1-X_2)\mathfrak{Z}(X_2-X_1).
\end{split}
\end{align}
Adding the two above equations \cref{eq:ZV-depth2,eq:ZV-depth2-2}, we have
\begin{align*}
\mathfrak{Z}^{\vee}(-X_1,X_2-X_1,X_3-X_1) + \mathfrak{Z}^{\vee}(X_1,X_2,X_3)
&=  X_1\mathfrak{Z}(X_2,X_1)
+ (X_2-X_1)\mathfrak{Z}(X_1-X_2)\mathfrak{Z}(X_2)
\\
&- (X_2-X_1)\mathfrak{Z}(X_1-X_2)\mathfrak{Z}(X_2-X_1)
- X_2\mathfrak{Z}(-X_2)\mathfrak{Z}(X_2)
\\
&+ X_2\mathfrak{Z}(-X_2)\mathfrak{Z}(X_2-X_1)
- X_1\mathfrak{Z}(X_2-X_1,-X_1).
\end{align*}
This equation corresponds to \cref{eq:relation7.5}. Note that the right-hand side does not depend on $X_3$, and thus we set $X_3=X_1$ for simplicity. Then, using \cref{eq:relation3.4} with $X_3=0$, we can rewrite the above equation as follows:
\begin{align*}
-X_1\mathfrak{Z}(X_2,X_1)
&=  -X_1\mathfrak{Z}(X_1,X_2)
- \mathfrak{Z}^{\vee}(X_1,X_2,X_1)
- (X_2-X_1)\mathfrak{Z}(X_1-X_2)\mathfrak{Z}(X_2-X_1)
\\
&+ (X_2-X_1)\mathfrak{Z}(X_1-X_2)\mathfrak{Z}(X_2)
+ (X_2-X_1)\mathfrak{Z}(-X_2)\mathfrak{Z}(X_2-X_1)
\\
&+ X_1\mathfrak{Z}(X_1)\mathfrak{Z}(-X_2)
- X_2\mathfrak{Z}(-X_2)\mathfrak{Z}(X_2)
+ X_1\mathfrak{Z}^S(X_1,X_2).
\end{align*}
Therefore, replacing \(X_2\) by \(X_2-X_1\) and \(X_1\) by \(X_3-X_1\) in the above equation, and using \cref{eq:relation3.5}, we obtain
\begin{align*}
\mathfrak{A}_2(X_1,X_2,X_3)
&= (X_1-X_3)\mathfrak{Z}(X_3-X_1,X_2-X_1)
- \mathfrak{Z}^{\vee}(X_3-X_1,X_2-X_1,X_3-X_1)
\\
&- (X_2-X_3)\mathfrak{Z}(X_3-X_2)\mathfrak{Z}(X_2-X_3)
+ (X_2-X_3)\mathfrak{Z}(X_1-X_2)\mathfrak{Z}(X_2-X_3)
\\
&+ (X_3-X_1)\mathfrak{Z}(X_3-X_1)\mathfrak{Z}(X_1-X_2)
- (X_2-X_1)\mathfrak{Z}(X_1-X_2)\mathfrak{Z}(X_2-X_1)
\\
&+ (X_3-X_1)\mathfrak{Z}^S(X_3-X_1,X_2-X_1).
\end{align*}
As a consequence, we obtain the following two expressions.
\begin{align*}
(X_1-X_3)\mathfrak{Z}(X_2-X_1,X_3-X_1)
&= \mathfrak{A}_2(X_1,X_2,X_3)
+(X_2-X_3)\mathfrak{Z}(X_3-X_2)\mathfrak{Z}(X_2-X_1)
\intertext{and}
(X_1-X_2)\mathfrak{Z}(X_2-X_1,X_3-X_1)
&= \mathfrak{A}_2(X_1,X_3,X_2)
+ \mathfrak{Z}^{\vee}(X_2-X_1,X_3-X_1,X_2-X_1)
\\
&- (X_2-X_1)\mathfrak{Z}^S(X_2-X_1,X_3-X_1)
\\
&- (X_3-X_2)\mathfrak{Z}(X_1-X_3)\mathfrak{Z}(X_3-X_2)
- (X_2-X_1)\mathfrak{Z}(X_2-X_1)\mathfrak{Z}(X_1-X_3)
\\
&+ (X_3-X_1)\mathfrak{Z}(X_1-X_3)\mathfrak{Z}(X_3-X_1)
+ (X_3-X_2)\mathfrak{Z}(X_2-X_3)\mathfrak{Z}(X_3-X_2).
\end{align*}
Therefore, we have
\begin{align*}
F_3 &\ceq \left(\frac{\partial}{\partial X_2}+\frac{\partial}{\partial X_3}\right)
\frac{\partial}{\partial X_3}(X_3-X_1)\mathfrak{Z}(X_2-X_1,X_3-X_1)
\\
&= \frac{\partial}{\partial X_1}\frac{\partial}{\partial X_3}(X_1-X_3)\mathfrak{Z}(X_2-X_1,X_3-X_1)
\\
&= \frac{\partial}{\partial X_1}\frac{\partial}{\partial X_3}
(\mathfrak{A}_2(X_1,X_2,X_3)
-(X_3-X_2)\mathfrak{Z}(X_3-X_2)\mathfrak{Z}(X_2-X_1))
\intertext{and}
F_2 &\ceq \left(\frac{\partial}{\partial X_2}+\frac{\partial}{\partial X_3}\right)
\frac{\partial}{\partial X_2}(X_2-X_1)\mathfrak{Z}(X_2-X_1,X_3-X_1)
\\
&= \frac{\partial}{\partial X_1}\frac{\partial}{\partial X_2}(X_1-X_2)\mathfrak{Z}(X_2-X_1,X_3-X_1)
\\
&= \frac{\partial}{\partial X_1}\frac{\partial}{\partial X_2}
(\mathfrak{A}_2(X_1,X_3,X_2)
+ \mathfrak{Z}^{\vee}(X_2-X_1,X_3-X_1,X_2-X_1)
- (X_2-X_1)\mathfrak{Z}^S(X_2-X_1,X_3-X_1)
\\
&- (X_3-X_2)\mathfrak{Z}(X_3-X_2)\mathfrak{Z}(X_1-X_3)
- (X_2-X_1)\mathfrak{Z}(X_2-X_1)\mathfrak{Z}(X_1-X_3)).
\end{align*}
From \cref{lem:yasuda-lem}, we have
$$\frac{\partial}{\partial X_3}\mathfrak{Z}(X_2,X_3) = F_3|_{X_1=0} -\frac{1}{\wt} \frac{\partial}{\partial X_3}(X_3(F_2 + F_3)|_{X_1=0}),$$
where $\wt$ denotes the weight of multiple zeta value. Expanding these generating functions, we obtain
\begin{align*}
\zeta(l_2,l_3+1)
&= \frac{l_2+1}{l_2+l_3+1}\Biggl(
(-1)^{l_2+l_3+1}\zeta_S^{\sha}((1,l_2,l_3)^{\vee})
+\sum_{\substack{k_1,k_3\geq 1\\k_1+k_3=1+l_2+l_3}}
(-1)^{k_3+l_3}\zeta^{\sha}(k_1)\zeta^{\sha}(k_3)(k_1-1)\binom{k_3}{l_3}\Biggr)
\\
&- \frac{l_2}{l_2+l_3+1}\Biggl(
(-1)^{1+l_2+l_3}\zeta_S^{\sha}((1,l_3,l_2)^{\vee})
- \sum_{\substack{k_1,k_2\geq 1\\k_1+k_2=1+l_2+l_3}}
(-1)^{k_1+l_2}\zeta^{\sha}(k_1)\zeta^{\sha}(k_2)(k_1-1)\binom{k_2}{l_2}
\\
&- \sum_{\substack{k_1,k_2,k_3\geq 1\\k_1+k_2+k_3=1+l_2+l_3}}
(-1)^{k_1+k_2+k_3}\zeta_S^{\sha}((k_1,k_2,k_3)^{\vee})
\binom{k_1+k_3}{l_2}\binom{k_2-1}{l_3-1}
\\
&+ \sum_{\substack{k_2,k_3\geq 1\\k_2+k_3=1+l_2+l_3}}
((-1)^{k_2+l_2+l_3}\zeta^{\sha}(k_2)\zeta^{\sha}(k_3)+\zeta_S^{\sha}(k_2,k_3))
\binom{k_2}{l_2}\binom{k_3-1}{l_3-1}
\Biggr).
\end{align*}
The above algorithm can be carried out for arbitrary depth. In the case of depth 2, a slightly simpler explicit formula can be obtained as an element of $\mathcal{Z}/\zeta(2)\mathcal{Z}$ in the next section.

\section{The depth 2 case revisited}\label{sec:depth2-rev}


\begin{lem}\label{lem:ZS-dual-depth3} For positive integers $k_1$, $k_2$, and $k_3$, we have
\begin{align*}
\zeta_S((k_1,k_2,k_3)^{\vee})
&= \zeta_S^{\star}(k_1,k_2,k_3)
- (-1)^{k_1+k_3}\sum_{i=1}^{k_2-1}
\binom{k_1+i}{k_1}\binom{k_2+k_3-i}{k_3}Z_S(k_1+i)Z_S(k_2+k_3-i).
\end{align*}
In particular, we have $\zeta_S((k_1,1,k_3)^{\vee}) = \zeta_S^{\star}(k_1,1,k_3)$.
\end{lem}

\begin{proof}
Assume first that $k_2\geq 2$. Using \cref{lem:basic-prop} and \cref{eq:OSS}, we obtain
\begin{align*}
\zeta_S((k_1,k_2,k_3)^{\vee})
&= \zeta_S(\{1\}^{k_1-1},2,\{1\}^{k_2-2},2,\{1\}^{k_3-1})
\\
&\overset{\cref{lem:repeated-index} \text{ and } \cref{lem:stuffle-antipode}}{=}
(-1)^{k_1+k_2+k_3-1}\zeta^{\star}_S(\{1\}^{k_3-1},2,\{1\}^{k_2-2},2,\{1\}^{k_1-1})
\\
&\quad\quad + (-1)^{k_1+k_2+k_3-1}\sum_{i=1}^{k_2-1}
(-1)^{k_1+i-1}\zeta_S(\{1\}^{k_1-1},2,\{1\}^{i-1})
\zeta^{\star}_S(\{1\}^{k_3-1},2,\{1\}^{k_2-1-i})
\\
&\overset{\cref{eq:OSS},\, \cref{lem:anti-index}\text{ and }\cref{lem:Hoffman-dual}}{=} \zeta_S^{\star}(k_1,k_2,k_3)
- (-1)^{k_1+k_3}\sum_{i=1}^{k_2-1}
\binom{k_1+i}{k_1}\binom{k_2+k_3-i}{k_3}Z_S(k_1+i)Z_S(k_2+k_3-i).
\end{align*}
The case $k_2=1$ follows by a similar argument. This completes the proof.
\end{proof}

\begin{lem}\label{lem:depth3-0-formula}
For any integer $k>0$, we have
\begin{align*}
\sum_{\substack{r+s=k+1\\r,s\geq 1}}\zeta_S(r,s,1) = (1+(-1)^{k-1})\zeta_S(1,1,k).
\end{align*}
Furthermore, for any integer $k>0$, we have
\begin{align*}
\sum_{\substack{r+s=k+1\\r,s\geq 1}}\zeta_S(r,1,s) = 0.
\end{align*}
\end{lem}

\begin{proof} The depth-three case of the equation in (2) of \cref{thm:KZ} can be written as
\begin{align*}
\mathfrak{Z}^S(X,-Z,-Y) &= \mathfrak{Z}^S(X,X+Y,X+Z) + \mathfrak{Z}^S(Y,X+Y,X+Z) + \mathfrak{Z}^S(Y,Z,X+Z).
\end{align*}
Substituting $X=0$ and $Z=0$, we obtain
\begin{align*}
\mathfrak{Z}^S(0,0,-Y) &= \mathfrak{Z}^S(0,Y,0) + \mathfrak{Z}^S(Y,Y,0) + \mathfrak{Z}^S(Y,0,0).
\end{align*}
Furthermore, from (1) of \cref{thm:KZ}, we have 
$$0=\mathfrak{Z}^S(0,0)\mathfrak{Z}^S(Y)=\mathfrak{Z}^S(Y,0,0)+\mathfrak{Z}^S(0,Y,0)+\mathfrak{Z}^S(0,0,Y)\mod\zeta(2).$$
Thus, we obtain
\begin{align*}
\mathfrak{Z}^S(Y,Y,0) = \mathfrak{Z}^S(0,0,Y) + \mathfrak{Z}^S(0,0,-Y) \mod \zeta(2).
\end{align*}
Expanding these generating functions on both sides and comparing the coefficients, we obtain the desired first formula. If $k$ is odd, the second formula follows immediately from \cref{lem:anti-index}. If $k$ is even, then the right-hand side of the first formula vanishes.
Now, summing each term in the identity
$$ 0 = \zeta_S(1)\zeta_S(r,s) = \zeta_S(1,r,s)+\zeta_S(r,1,s)+\zeta_S(r,s,1) $$
over all $r,s\geq1$ with $r+s=k+1$, the sums of the first and third terms vanish by the first formula and \cref{lem:anti-index}. Hence, the sum of the second term also vanishes, which proves the second formula. This completes the proof.
\end{proof}

By substituting $(X_1,X_2,X_3)$ with $(-X_1,0,X_3-X_1)$ in \cref{eq:ZV-depth2}, and then partially differentiating with respect to $X_3$, we obtain
\begin{align*}
\frac{\partial}{\partial X_3}\mathfrak{Z}^{\vee}(-X_1,0,X_3-X_1)
&= \frac{\partial}{\partial X_3}(-X_3\mathfrak{Z}(X_1,X_3)
+ (X_3-X_1)\mathfrak{Z}(X_3-X_1)\mathfrak{Z}(X_1)
+ (X_3-X_1)\mathfrak{Z}(0,X_3-X_1)).
\end{align*}
On the other hand, substituting $(X,Y)$ with $(0,X_3-X_1)$ into the identity $\mathfrak{Z}(X)\mathfrak{Z}(Y) = \mathfrak{Z}(X,X+Y) + \mathfrak{Z}(Y,X+Y)$, we have
$$0 = \mathfrak{Z}(0,X_3-X_1) + \mathfrak{Z}(X_3-X_1,X_3-X_1).$$
The above two equations imply
\begin{align*}
\frac{\partial}{\partial X_3}X_3\mathfrak{Z}(X_1,X_3)
&= \frac{\partial}{\partial X_3}(-\mathfrak{Z}^{\vee}(-X_1,0,X_3-X_1)
\\
&\quad + (X_3-X_1)\mathfrak{Z}(X_3-X_1)\mathfrak{Z}(X_1)
- (X_3-X_1)\mathfrak{Z}(X_3-X_1,X_3-X_1)).
\end{align*}
Expanding this equation, we obtain
\begin{align*}
\zeta^{\sha}(k_1,k_2)
&= \sum_{\substack{l_1+l_2=k_1+k_2\\l_1\geq 2,\, l_2\geq 1}}
\binom{l_2}{k_2}((-1)^{l_1+l_2+k_1}\zeta_S^{\sha}((l_1-1,1,l_2)^{\vee}) + (-1)^{l_2+k_2}\zeta^{\sha}(l_1)\zeta^{\sha}(l_2))
\\
&\quad + (-1)^{k_1}\binom{k_1+k_2-1}{k_2}\sum_{\substack{l_1+l_2=k_1+k_2\\l_1,l_2\geq 1}}\zeta^{\sha}(l_1,l_2).
\end{align*}
Here, for even $k>1$, by \cref{eq:depth1StoZ} and the definition of $\Sha$-symmetric multiple zeta value, we have
$$0 = \sum_{n=1}^{k-1}(-1)^{n-1}\binom{k}{n}Z_S(k)
= \sum_{n=1}^{k-1}\zeta_S(n,k-n)
= 2\sum_{n=1}^{k-1}Z_S(n,k-n)
+\sum_{n=1}^{k-1}(-1)^{n}Z_S(n)Z_S(k-n).$$
Therefore, if $k_1+k_2$ is even, we have
\begin{align*}
Z_S(k_1,k_2)
&= (-1)^{k_1}\sum_{\substack{l_1+l_2=k_1+k_2\\l_1\geq 2,\, l_2\geq 1}}
\left(\binom{l_2}{k_2}(\zeta_S((l_1-1,1,l_2)^{\vee}) - Z_S(l_1)Z_S(l_2))
+ \frac{1}{2}\binom{k_1+k_2-1}{k_2}Z_S(l_1)Z_S(l_2)\right).
\end{align*}
From \cref{lem:ZS-dual-depth3,lem:depth3-0-formula}, together with the fact that $\zeta_S^{\star}(k_1,k_2,k_3)=\zeta_S(k_1,k_2,k_3)$ when $k_1+k_2+k_3$ is even, we obtain the following expression.

\begin{thm}\label{thm:DZ-SMZV}
For any integers $k_1,k_2\geq 1$ such that $k_1+k_2$ is even, we have
\begin{align*}
Z_S(k_1,k_2)
&= (-1)^{k_1}\sum_{\substack{l_1+l_2=k_1+k_2\\l_1\geq 2,\, l_2\geq 1}}
\left(\binom{l_2}{k_2}- \frac{1}{2} \binom{k_1+k_2-1}{k_2}\right)
(\zeta_S(l_1-1,1,l_2)-Z_S(l_1)Z_S(l_2)).
\end{align*}
\end{thm}

\section{The space of triple SMZVs and FMZVs}\label{sec:depth3}

In this section, we consider the $\Q$-vector space spanned by triple SMZVs and FMZVs
$$\mathcal{TZ}^{\mathcal{F}}_k \ceq \Span_{\Q}\{\zeta_{\mathcal{F}}(k_1,k_2,k_3)\mid k_1+k_2+k_3=k\}\quad (\mathcal{F}=S\text{ or }\mathcal{A}).$$
If \(k\) is an odd integer, this structure is already known.
\begin{thm}[{\cite{Hof15,Zha08} for $\mathcal{F}=\mathcal{A}$, \cite{OSS21} for $\mathcal{F}=\mathcal{S}$}]
For any odd integer $k>1$ and $\mathcal{F}\in\{S,\mathcal{A}\}$, we have $\mathcal{TZ}^{\mathcal{F}}_k = \Q\cdot Z_{\mathcal{F}}(k)$.
\end{thm}
We investigate the structure of $\mathcal{TZ}^{\mathcal{F}}_k$ for an even integer $k$. In view of \Cref{conj-KZ}, we expect that all relations among SMZVs also hold among FMZVs. Indeed, properties \textup{(1)} and \textup{(2)} in \cref{thm:KZ} are known to hold also in $\mathcal{Z}_{\mathcal{A}}$:
\begin{thm}[Proposition 4 and Theorem 2 in \cite{KZ}]\label{thm:KZ-A}
We have the following.
\begin{enumerate}
\item $\zeta_{\mathcal{A}}(\mathds{k})$ satisfies the stuffle product.
\item For any $\mathds{k},\mathds{k}'\in\mathcal{I}$, we have $\zeta_{\mathcal{A}}(\mathds{k}\,\Sha\,\mathds{k}') = (-1)^{\wt(\mathds{k}')}\zeta_{\mathcal{A}}(\mathds{k},\overleftarrow{\mathds{k}'})$.
\end{enumerate}
\end{thm}
Furthermore, the corresponding identities
\cref{eq:OSS}--\cref{eq:depth1StoZ}
are known to hold for finite multiple zeta values; see
\cite[Theorem~3.8 and Proposition~2.5]{OSS21},
\cite[Theorems~4.4--4.6]{Hof15},
and \cite[(7.2)]{Kan19}.
Thus, the same argument shows that the analogues of
\cref{lem:ZS-dual-depth3,lem:depth3-0-formula}
also hold with $\mathcal S$ replaced by $\mathcal A$.
Here, the element $Z_{\mathcal{A}}$ corresponding to $Z_S$ is defined as follows. For $k\geq 2$, set
$$Z_{\mathcal{A}}(k)\ceq\left(\frac{B_{p-k}}{k}\bmod p\right)_p,$$
and set $Z_{\mathcal{A}}(1)\coloneqq 0$, where $B_k$ denotes the $k$-th Bernoulli number, defined by
$$\frac{t}{e^t-1} = \sum_{k\geq 0} B_k\frac{t^k}{k!}.$$
Note that $Z_{\mathcal{A}}(k)=0$ for even $k>0$.

\begin{prop}\label{lem:sym-x-one-y}
For any integer $k>1$ and $\mathcal{F}\in\{S,\mathcal{A}\}$, define
$$\mathcal{TZ}^{\vee,\mathcal{F}}_k \ceq \Span_{\Q}\{\zeta_{\mathcal{F}}((k_1,k_2,k_3)^{\vee})\mid k_1+k_2+k_3=k\}.$$
Then, for an integer $k>1$, we have
$$\mathcal{TZ}^{\vee,\mathcal{F}}_k = \Span_{\Q}\{\zeta_{\mathcal{F}}((r,1,s)^{\vee})\mid r+s+1=k\}.$$
\end{prop}

\begin{proof}
We present two proofs. The first applies only to SMZVs, while the second applies to both SMZVs and FMZVs.

(The first proof) Using \cref{eq:ZV-depth2}, we can directly verify
\begin{equation}\label{eq:w-to-1ab}
\frac{\partial}{\partial X_1}\frac{\partial}{\partial X_3}\mathfrak{Z}^{\vee}(X_1,X_2,X_3)
= \frac{\partial}{\partial X_1}\frac{\partial}{\partial X_3}\mathfrak{Z}^{\vee}(X_1-X_2,0,X_3-X_2).
\end{equation}
Comparing coefficients on both sides, we obtain
\begin{align}\label{eq:dual-to-1}
\zeta_S((k_1,k_2,k_3)^{\vee})
&= (-1)^{k_2-1}\sum_{n_1+1+n_3=k_1+k_2+k_3}\binom{n_1}{k_1}\binom{n_3}{k_3}\zeta_S((n_1,1,n_3)^{\vee}).
\end{align}
Hence, this statement for SMZVs is obtained.

(The second proof) It suffices to show that \cref{eq:w-to-1ab} holds. Equation \cref{eq:w-to-1ab} is equivalent to
\begin{equation*}
\frac{\partial}{\partial X_1}\frac{\partial}{\partial X_3}\mathfrak{Z}^{\vee}(X_1+X_2,X_2,X_3+X_2)
= \frac{\partial}{\partial X_1}\frac{\partial}{\partial X_3}\mathfrak{Z}^{\vee}(X_1,0,X_3).
\end{equation*}
By comparing coefficients in this identity, we are reduced to proving the following statement. For $k_1,k_2,k_3\geq 1$ and $\mathcal{F}\in\{S,\mathcal{A}\}$, define
\begin{align*}
V_{\mathcal{F}}(k_1,k_2,k_3)
\ceq \sum_{\substack{n_1,n_2,n_3\geq 1\\n_1+n_2+n_3=k_1+k_2+k_3}}
\binom{n_1}{k_1}\binom{n_3}{k_3} \zeta_{\mathcal{F}}((n_1,n_2,n_3)^{\vee}).
\end{align*}
Then, for $k_1,k_2,k_3\geq 1$ with $k_2\geq 2$, we have $V_{\mathcal{F}}(k_1,k_2,k_3)=0$. It suffices to prove that for $n\geq 2$,
\begin{equation}\label{eq:ro-show-V}
V_{\mathcal{F}}(k_1,n,k_3) = \zeta_{\mathcal{F}}((\{1\}^{k_1-1},3,\{1\}^{k_3-1})\Sha(\{1\}^{n-1}))
+ \zeta_{\mathcal{F}}((\{1\}^{k_1-1},3,\{1\}^{k_3})\Sha(\{1\}^{n-2})),
\end{equation}
since the two terms on the right-hand side cancel by property (2) in \cref{thm:KZ-A}. Under the correspondence $k\mapsto yx^{k-1}$, we can write the right-hand side of \cref{eq:ro-show-V} as
$$\mathrm{RHS} = \zeta_{\mathcal{F}}((y^{k_1},xx,y^{k_3-1})\Sha(y^{n-1}))
+ \zeta_{\mathcal{F}}((y^{k_1},xx,y^{k_3})\Sha(y^{n-2})).$$
By definition of the shuffle product, we have
\begin{align*}
\mathrm{RHS} &= \sum_{\substack{l_1,l_2,l_3\geq 0\\l_1+l_2+l_3=n-1}}\binom{k_1+l_1}{k_1}\binom{k_3+l_3-1}{k_3-1}\zeta_{\mathcal{F}}(y^{k_1+l_1},xy^{l_2}x,y^{k_3+l_3-1})
\\
&+ \sum_{\substack{l_1,l_2,l_3\geq 0\\l_1+l_2+l_3=n-2}}\binom{k_1+l_1}{k_1}\binom{k_3+l_3}{k_3}\zeta_{\mathcal{F}}(y^{k_1+l_1},xy^{l_2}x,y^{k_3+l_3}).
\end{align*}
By replacing $l_3$ by $l_3-1$ in the second summation and using the identity
$$\binom{k_3+l_3-1}{k_3-1}+\binom{k_3+l_3-1}{k_3} = \binom{k_3+l_3}{k_3},$$
we obtain
\begin{align*}
\mathrm{RHS} &= \sum_{\substack{l_1,l_2,l_3\geq 0\\l_1+l_2+l_3=n-1}}\binom{k_1+l_1}{k_1}\binom{k_3+l_3}{k_3}\zeta_{\mathcal{F}}(y^{k_1+l_1},xy^{l_2}x,y^{k_3+l_3-1})
\\
&= \sum_{\substack{l_1,l_2,l_3\geq 0\\l_1+l_2+l_3=n-1}}\binom{k_1+l_1}{k_1}\binom{k_3+l_3}{k_3}\zeta_{\mathcal{F}}((k_1+l_1,l_2+1,k_3+l_3)^{\vee})
= V_{\mathcal{F}}(k_1,n,k_3).
\end{align*}
Thus, \cref{eq:ro-show-V} holds, and the proof is complete.
\end{proof}

\begin{prop}\label{prop:ZZ-tiple}
For an even integer $k>1$, an integer $2\leq b\le k-1$, and $\mathcal{F}\in\{S,\mathcal{A}\}$, we have
\begin{align}\label{eq:ZZ-formula-1}
Z_{\mathcal{F}}(b)Z_{\mathcal{F}}(k-b) = \sum_{r=b-1}^{k-2}(-1)^r\frac{1}{b(k-b)}\binom{r}{b-1}B_{r-b+1}(\zeta_{\mathcal{F}}((r,k-r-1,1)^{\vee}) - \zeta_{\mathcal{F}}(r,k-r-1,1))
\end{align}
and
\begin{equation}\label{eq:ZZ-formula-2}
\begin{split}
Z_{\mathcal{F}}(b)Z_{\mathcal{F}}(k-b)
= (-1)^{b}&\sum_{r=b}^{k-1}(-1)^{r}\frac{k-r}{b(k-b)}\binom{r-1}{b-1}B_{r-b}\zeta_{\mathcal{F}}(r-1,1,k-r)
\\
&+ \sum_{r=b}^{k-1}(-1)^{r}\frac{1}{b(k-b)}\binom{r-1}{b-1}B_{r-b}\zeta_{\mathcal{F}}(r-1,k-r,1).
\end{split}
\end{equation}
\end{prop}

\begin{proof}
From \cref{lem:ZS-dual-depth3}, for an even integer $k$ and an integer $1\leq r\le k-2$ we have
\begin{align}\label{eq:sp-depth3}
\zeta_{\mathcal{F}}((r,k-r-1,1)^{\vee})
&= \zeta_{\mathcal{F}}(r,k-r-1,1)
+ (-1)^{r}\sum_{a=1}^{k-1}(1-\delta_{r,a})\binom{a}{r}(k-a) Z_{\mathcal{F}}(a)Z_{\mathcal{F}}(k-a).
\end{align}
Using \cref{eq:sp-depth3}, the right-hand side of \cref{eq:ZZ-formula-1} can be written as
\begin{align}\label{eq:tochu-ZZ}
\frac{1}{(k-b)}\sum_{a=1}^{k-1}\frac{a!}{b!}(k-a) Z_{\mathcal{F}}(a)Z_{\mathcal{F}}(k-a)\sum_{r=b-1}^{a}\frac{1}{(a-r)!(r-b+1)!}B_{r-b+1}
(1-\delta_{r,a})
\end{align}
By replacing $r$ by $r+b-1$ in \cref{eq:tochu-ZZ}, we have
\begin{align}\label{eq:tochu-ZZ2}
\frac{1}{(k-b)}\sum_{a=1}^{k-1}\frac{a!}{b!}(k-a) Z_{\mathcal{F}}(a)Z_{\mathcal{F}}(k-a)\sum_{r=0}^{a-b}\frac{1}{(a-b-r+1)!r!}B_{r}
\end{align}
Here, by definition of the Bernoulli numbers, we have
$$t = \frac{t}{e^t-1}(e^t-1) = \left(\sum_{n=0}^{\infty}B_{n}\frac{t^n}{n!}\right)\left(t\sum_{n=0}^{\infty}\frac{t^n}{(n+1)!}\right)
= t\sum_{k=0}^{\infty}\sum_{n=0}^{k}\frac{B_n}{n!(k-n+1)!}t^k$$
Thus, by comparing the coefficients, we have
\begin{equation}\label{eq:Bernouill-rel}
\sum_{n=0}^{k}\frac{B_n}{n!(k-n+1)!} = \left\{\begin{array}{cc}1 & (k = 0) \\ 0 & (k > 0) \end{array}\right. .
\end{equation}
By \cref{eq:Bernouill-rel}, the inner sum in \cref{eq:tochu-ZZ2} vanishes unless $a=b$, in which case it is equal to $1$. Therefore, \cref{eq:tochu-ZZ2} reduces to
$Z_{\mathcal{F}}(b)Z_{\mathcal{F}}(k-b)$. Thus, we obtain \cref{eq:ZZ-formula-1}. Moreover, from \cref{eq:dual-to-1,eq:ZZ-formula-1}, we have
\begin{equation}\label{eq:tochu-ZZ3}
\begin{split}
Z_{\mathcal{F}}(b)Z_{\mathcal{F}}(k-b) &= \sum_{r=b-1}^{k-2}\frac{1}{b(k-b)}\binom{r}{b-1}B_{r-b+1}\sum_{c=1}^{k-2}\binom{c}{r}(k-1-c)\zeta_{\mathcal{F}}(c,1,k-1-c)
\\
&- \sum_{r=b-1}^{k-2}(-1)^r\frac{1}{b(k-b)}\binom{r}{b-1}B_{r-b+1}\zeta_{\mathcal{F}}(r,k-r-1,1).
\end{split}
\end{equation}
Here, using the identity
$$\binom{a}{k}\binom{k}{b} = \binom{a}{b}\binom{a-b}{k-b},$$
the first line on the right-hand side of \cref{eq:tochu-ZZ3} can be written as
\begin{align}\label{eq:tochu-ZZ4}
\sum_{c=1}^{k-2}\frac{k-1-c}{b(k-b)}\binom{c}{b-1}\zeta_{\mathcal{F}}(c,1,k-1-c)\sum_{r=b-1}^{c}B_{r-b+1}\binom{c-b+1}{r-b+1}.
\end{align}
Furthermore, by using the identity
$$\sum_{r=0}^{k}\binom{k}{r}B_{r} = (-1)^kB_k,$$
which is obtained from \cref{eq:Bernouill-rel}, \cref{eq:tochu-ZZ4} can be written as
\begin{align}\label{eq:tochu-ZZ5}
\sum_{c=1}^{k-2}(-1)^{c-b+1}B_{c-b+1}\frac{k-1-c}{b(k-b)}\binom{c}{b-1}\zeta_{\mathcal{F}}(c,1,k-1-c).
\end{align}
Replacing the first line in right-hand side of \cref{eq:tochu-ZZ3} by \cref{eq:tochu-ZZ5} and making a simple change of variables, we obtain \cref{eq:ZZ-formula-2}.
\end{proof}

\begin{prop}\label{thm:triple-conj}
For even $k>1$ and $\mathcal{F}\in\{S,\mathcal{A}\}$, we have
\begin{equation*}
\mathcal{TZ}^{\mathcal{F}}_k = \Span_{\Q}\{\zeta_{\mathcal{F}}(1,r,k-r-1)\mid 0<r<k-1\}.
\end{equation*}
\end{prop}

\begin{proof}
By \cref{lem:ZS-dual-depth3}, $\zeta_{\mathcal{F}}(k_1,k_2,k_3)\in \mathcal{TZ}^{\mathcal{F}}_k$ can be expressed as a $\Q$-linear combination of $\zeta_{\mathcal{F}}((k_1,k_2,k_3)^{\vee})$ and $\{Z_{\mathcal{F}}(r)Z_{\mathcal{F}}(k-r)\mid 1\leq r\leq k-1\}$. Furthermore, using \cref{lem:sym-x-one-y,prop:ZZ-tiple} and noting that
$$0 = \zeta_{\mathcal{F}}(1)\zeta_{\mathcal{F}}(r,s) = \zeta_{\mathcal{F}}(1,r,s)+\zeta_{\mathcal{F}}(r,1,s)+(-1)^{r+s+1}\zeta_{\mathcal{F}}(1,s,r),$$
we find that $\zeta_{\mathcal{F}}((k_1,k_2,k_3)^{\vee})$ and $Z_{\mathcal{F}}(r)Z_{\mathcal{F}}(k-r)$ can be expressed as a $\Q$-linear combination of $\{\zeta_{\mathcal{F}}(1,r,s)\mid r+s+1=k\}$.
\end{proof}

Kaneko--Zagier \cite[Conjecture~1]{KZ} conjectured that $\mathcal{TZ}^{\mathcal{F}}_k$ admits a smaller generating set than the one given in \cref{thm:triple-conj}; more precisely, they conjectured that
\begin{equation*}
\mathcal{TZ}^{\mathcal{F}}_k \overset{?}{=} \Span_{\mathbb{Q}}
\{\zeta_{\mathcal{F}}(1,r,k-r-1)\mid 0<r<k/2-1\}.
\end{equation*}
Furthermore, they conjectured a family of relations satisfied by these generators and also conjectured that these relations exhaust all relations among them.

Motivated by \cref{thm:DZ-SMZV}, we define $Z_{\mathcal{A}}(k_1,k_2)$ by the following
\begin{definition}
For any integers $k_1,k_2\geq 1$ such that $k_1+k_2$ is even, we define
\begin{align}\label{eq:def-DZ-FMZV}
Z_{\mathcal{A}}(k_1,k_2)
&\ceq (-1)^{k_1}\sum_{\substack{l_1+l_2=k_1+k_2\\l_1\geq 2,\, l_2\geq 1}}
\left(\binom{l_2}{k_2}- \frac{1}{2} \binom{k_1+k_2-1}{k_2}\right)
(\zeta_{\mathcal{A}}(l_1-1,1,l_2)-Z_{\mathcal{A}}(l_1)Z_{\mathcal{A}}(l_2)).
\end{align}
\end{definition}

Kaneko--Zagier \cite{KZ} introduced $Z_{\mathcal{A}}(k_1,k_2)$ by a different definition arising from a heuristic argument (see \Cref{sec:def-DFMZV}), and conjectured \cref{eq:def-DZ-FMZV} as a relation satisfied by this $Z_{\mathcal{A}}(k_1,k_2)$. For $\mathcal{F}\in\{S,\mathcal{A}\}$, we denote
$$\mathcal{DZ}_k^{(\mathcal{F})}\ceq\Span_{\Q}\{Z_{\mathcal{F}}(r,k-r)\mid 1\leq r\leq k-1\}.$$

\begin{lem}\label{lem:DZ-shuffle}
For integers $r,s\geq 1$ such that $k=r+s$ is even and $\mathcal{F}\in\{S,\mathcal{A}\}$, we have
\begin{equation}\label{eq:depth-2-shuffle}
Z_{\mathcal{F}}(r)Z_{\mathcal{F}}(s) = \sum_{i=1}^{k-1}\left(\binom{i-1}{r-1}+\binom{i-1}{s-1}\right)Z_{\mathcal{F}}(k-i,i).
\end{equation}
\end{lem}

\begin{proof}
For integers $r\geq 2$ and $s\geq 1$, we set $a_{r,s} \ceq \zeta_{\mathcal{F}}(r-1,1,s)-Z_{\mathcal{F}}(r)Z_{\mathcal{F}}(s)$ and $a_{1,s} \ceq 0$. Moreover, we define
$$F(X,Y) \ceq \sum_{\substack{r,s\geq 1\\r+s:\text{even}}}a_{r,s} X^{r-1}Y^{s-1}
\quad\text{ and }\quad
\mathfrak{Z}^{\mathcal{F}}_{\text{ev}}(X,Y) \ceq \sum_{\substack{r,s\geq 1\\r+s:\text{even}}}Z_{\mathcal{F}}(r,s) X^{r-1}Y^{s-1}.$$
By \cref{thm:DZ-SMZV} and definition of $Z_{\mathcal{A}}(r,s)$, we have
\begin{align*}
\mathfrak{Z}^{\mathcal{F}}_{\text{ev}}(X,Y) = -\frac{1}{Y}\left((-X+Y)F(-X,-X+Y) - \frac{1}{2}(-X+Y)F(-X+Y,-X+Y) + \frac{1}{2}XF(-X,-X)\right).
\end{align*}
By a straightforward calculation, we obtain
\begin{align*}
\mathfrak{Z}^{\mathcal{F}}_{\text{ev}}(X,X+Y) + \mathfrak{Z}^{\mathcal{F}}_{\text{ev}}(Y,X+Y)
&= - \frac{Y}{X+Y}F(-X,Y) - \frac{X}{X+Y}F(-Y,X).
\end{align*}
Here, we have
\begin{align*}
YF(-X,Y) + XF(-Y,X)
&= \sum_{\substack{r,s\geq 2\\r+s:\text{odd}}}(-1)^{r-1}(\zeta_{\mathcal{F}}(r-1,1,s-1)-Z_{\mathcal{F}}(r)Z_{\mathcal{F}}(s-1)) X^{r-1}Y^{s-1}
\\
&\quad\quad +\sum_{\substack{r,s\geq 2\\r+s:\text{odd}}}(-1)^{s-1}(\zeta_{\mathcal{F}}(s-1,1,r-1)-Z_{\mathcal{F}}(s)Z_{\mathcal{F}}(r-1)) X^{r-1}Y^{s-1}
\\
&= - \sum_{\substack{r,s\geq 2\\r+s:\text{odd}}}(Z_{\mathcal{F}}(r)Z_{\mathcal{F}}(s-1)+Z_{\mathcal{F}}(s)Z_{\mathcal{F}}(r-1)) X^{r-1}Y^{s-1}
\\
&= -(X+Y)\mathfrak{P}^{\mathcal{F}}_{\text{ev}}(X,Y),
\end{align*}
where
$$\mathfrak{P}^{\mathcal{F}}_{\text{ev}}(X,Y) \ceq \sum_{\substack{r,s\geq 1\\r+s:\text{even}}}Z_{\mathcal{F}}(r)Z_{\mathcal{F}}(s) X^{r-1}Y^{s-1}.$$
Therefore, we have
\begin{align*}
\mathfrak{P}^{\mathcal{F}}_{\text{ev}}(X,Y) = \mathfrak{Z}^{\mathcal{F}}_{\text{ev}}(X,X+Y) + \mathfrak{Z}^{\mathcal{F}}_{\text{ev}}(Y,X+Y).
\end{align*}
Comparing coefficients in this generating function, we obtain the desired formula.
\end{proof}

\begin{thm}\label{thm:T=D}
For any even integer $k>1$ and $\mathcal{F}\in\{S,\mathcal{A}\}$, we have $\mathcal{TZ}^{\mathcal{F}}_k = \mathcal{DZ}_k^{(\mathcal{F})}$.
\end{thm}

\begin{proof}
The case $k=2$ is immediate. We assume that $k\geq4$.
By \cref{prop:ZZ-tiple}, \cref{thm:DZ-SMZV},
and the definition of $Z_{\mathcal A}(r,s)$, we have
\begin{align*}
\mathcal{TZ}^{\mathcal{F}}_k
\supset \Span_{\Q}\{\zeta_{\mathcal{F}}(r-1,1,k-r)-Z_{\mathcal{F}}(r)Z_{\mathcal{F}}(k-r)\mid 2\leq r\leq k-1\}
\supset \mathcal{DZ}_k^{(\mathcal{F})}.
\end{align*}
Conversely, we prove that $\mathcal{TZ}^{\mathcal{F}}_k\subset \mathcal{DZ}_k^{(\mathcal{F})}$. Using \cref{lem:ZS-dual-depth3,lem:sym-x-one-y}, we see that it suffices to show that $\mathcal{DZ}_k^{(\mathcal{F})}$ contains both $\zeta_{\mathcal{F}}(r,1,k-r-1)$ and $Z_{\mathcal{F}}(r)Z_{\mathcal{F}}(k-r)$ for $1\leq r\leq k-2$. By \cref{lem:DZ-shuffle}, the product $Z_{\mathcal{F}}(r)Z_{\mathcal{F}}(k-r)$ belongs to $\mathcal{DZ}_k^{(\mathcal{F})}$. Moreover, combining this with \cref{lem:depth3-0-formula}, \cref{thm:DZ-SMZV}, and the definition of $Z_{\mathcal{A}}(r,s)$, we obtain
\begin{align*}
\sum_{b=1}^{k-2}\binom{b}{s} \zeta_{\mathcal{F}}(k-b-1,1,b) \in \mathcal{DZ}_k^{(\mathcal{F})}
\end{align*}
for $1\leq s\leq k-2$. Since the matrix $(\binom{b}{s})_{1\leq b,s\leq k-2}$ is invertible,  $\mathcal{DZ}_k^{(\mathcal{F})}$ contains $\zeta_{\mathcal{F}}(r,1,k-r-1)$.
\end{proof}

By \cref{lem:DZ-shuffle}, $\mathcal{DZ}_k^{(\mathcal{F})}$ satisfies the shuffle product relation \cref{eq:depth-2-shuffle}. When $\mathcal{F}=S$, it naturally also satisfies the stuffle product relation. While this paper was being prepared, Zhao proved in \cite{Zha26} that the stuffle product relation also holds for $\mathcal{F}=\mathcal{A}$: $Z_{\mathcal{A}}(r)Z_{\mathcal{A}}(s) = Z_{\mathcal{A}}(r,s)+Z_{\mathcal{A}}(s,r)$.
Therefore, \cref{prop:harm1,prop:harm2} follow immediately from \cite[Theorem~1]{GKZ06}. Nevertheless, we retain these propositions for the record.

\begin{prop}\label{prop:harm1}
For even $k>0$ and $\mathcal{F}\in\{S,\mathcal{A}\}$, we have
$$2Z_{\mathcal{F}}(k-1,1)
= \sum_{k_1+k_2=k}Z_{\mathcal{F}}(k_1)Z_{\mathcal{F}}(k_2)
= -2Z_{\mathcal{F}}(1,k-1).$$
In particular, we have
$$Z_{\mathcal{F}}(k-1,1) + Z_{\mathcal{F}}(1,k-1) = 0 \ \ (=Z_{\mathcal{F}}(1)Z_{\mathcal{F}}(k-1)).$$
\end{prop}

\begin{proof}
We use the notation and results from the proof of \cref{lem:DZ-shuffle}. Replacing $Y$ by $-X$ in
\begin{align*}
\mathfrak{P}^{\mathcal{F}}_{\mathrm{ev}}(X,Y)
&= \mathfrak{Z}^{\mathcal{F}}_{\mathrm{ev}}(X,X+Y) + \mathfrak{Z}^{\mathcal{F}}_{\mathrm{ev}}(Y,X+Y),
\end{align*}
we obtain
\begin{align*}
\mathfrak{P}^{\mathcal{F}}_{\text{ev}}(X,X) = \mathfrak{P}^{\mathcal{F}}_{\text{ev}}(X,-X) &= \mathfrak{Z}^{\mathcal{F}}_{\text{ev}}(X,0) + \mathfrak{Z}^{\mathcal{F}}_{\text{ev}}(-X,0)
= 2\mathfrak{Z}^{\mathcal{F}}_{\text{ev}}(X,0).
\end{align*}
This implies the first equality in the statement. Moreover, substituting $X=0$ in
\begin{align*}
\mathfrak{Z}^{\mathcal{F}}_{\text{ev}}(X,Y)
&= -\frac{1}{Y}\left((-X+Y)F(-X,-X+Y) - \frac{1}{2}(-X+Y)F(-X+Y,-X+Y) + \frac{1}{2}XF(-X,-X)\right),
\end{align*}
we obtain
\begin{align*}
\mathfrak{Z}^{\mathcal{F}}_{\text{ev}}(0,Y) 
&= \left(\frac{1}{2}F(Y,Y) - F(0,Y)\right)
= -\frac{1}{2}\mathfrak{P}^{\mathcal{F}}_{\text{ev}}(Y,Y).
\end{align*}
Here, the second equality follows from the definition of $F(X,Y)$ and \cref{lem:depth3-0-formula}. This implies the second equality in the statement. 
\end{proof}

\begin{prop}\label{prop:harm2}
For even $k>0$ and $\mathcal{F}\in\{S,\mathcal{A}\}$, we have
\begin{align*}
\sum_{r=1}^{k-2}Z_{\mathcal{F}}(r,k-r) = 0.
\end{align*}
\end{prop}

\begin{proof}
We use the notation and results from the proof of \cref{lem:DZ-shuffle}. Substituting $Y=0$ into
\begin{align*}
\mathfrak{P}^{\mathcal{F}}_{\mathrm{ev}}(X,Y)
&= \mathfrak{Z}^{\mathcal{F}}_{\mathrm{ev}}(X,X+Y) + \mathfrak{Z}^{\mathcal{F}}_{\mathrm{ev}}(Y,X+Y)
\end{align*}
and using \cref{prop:harm1}, we obtain
\begin{align*}
0 = \mathfrak{P}^{\mathcal{F}}_{\mathrm{ev}}(X,0)
&= \mathfrak{Z}^{\mathcal{F}}_{\mathrm{ev}}(X,X) + \mathfrak{Z}^{\mathcal{F}}_{\mathrm{ev}}(0,X)
= \mathfrak{Z}^{\mathcal{F}}_{\mathrm{ev}}(X,X) - \mathfrak{Z}^{\mathcal{F}}_{\mathrm{ev}}(X,0).
\end{align*}
This implies the equation in the statement. 
\end{proof}

We denote by $\mathcal{DZ}_k$ the $\mathbb{Q}$-vector space spanned by depth-two MZVs of weight $k$. This space has been studied in \cite{BS13,GKZ06,Kan04}, and its relation to modular forms is well known. The dimension of $\mathcal{DZ}_k$ is conjectured to be as follows:
\begin{equation*}
\dim_{\Q}\mathcal{DZ}_k \overset{?}{=} \left[\frac{k+1}{2}\right]-1-\dim S_k(SL_2(\Z))
\end{equation*}
for even $k>2$. This conjecture has been proved only in one direction: the right-hand side gives an upper bound for $\dim_{\mathbb{Q}}\mathcal{DZ}_k$. This upper bound also gives an upper bound for
$\dim_{\mathbb{Q}}\mathcal{DZ}_k^{(S)}$. Indeed, since
$\zeta(k)\in\mathcal{DZ}_k$ is mapped to $Z_S(k)=0$ for even $k>2$, we have $\dim_{\mathbb Q}\mathcal{DZ}^{(S)}_k \leq \dim_{\mathbb Q}\mathcal{DZ}_k-1$. Therefore, for even $k>0$, we have
$$\dim_{\Q}\mathcal{TZ}_k^{S} = \dim_{\Q}\mathcal{DZ}_k^{(S)}
\leq \left[\frac{k+1}{2}\right]-2-\dim S_k(\mathrm{SL}_2(\Z)).$$
This upper bound is conjectured to be sharp.
One can also obtain an upper bound for $\dim_{\mathbb{Q}}\mathcal{DZ}_k^{(\mathcal{A})}$ equal to that for $\dim_{\mathbb{Q}}\mathcal{DZ}_k^{(S)}$. Indeed, the depth-two double shuffle relations
\begin{equation}\label{eq:depth-2-DSR}
\zeta(r,s)+\zeta(s,r)+\zeta(r+s) = \zeta(r)\zeta(s) = \sum_{i=1}^{k-1}\left(\binom{i-1}{r-1}+\binom{i-1}{s-1}\right)\zeta(k-i,i),
\end{equation}
together with the formula
\begin{equation}\label{eq:depth-2-cusp}
\frac{B_rB_s}{r!s!}\zeta(k) + \frac{2B_k}{k!}\zeta(r)\zeta(s) = 0
\end{equation}
for even $r,s>0$ with $k=r+s$, which follows from the evaluation
$$\zeta(k) = -\frac{1}{2}\frac{(2\pi i)^k}{k!}B_k \quad \text{for even }k>0,$$
yield the conjecturally sharp upper bound for $\dim_{\mathbb{Q}}\mathcal{DZ}_k$ (see \cite[Theorem~3]{GKZ06}).
In $\mathcal{DZ}_k^{(\mathcal{A})}$, \cref{eq:depth-2-DSR} follows from \cref{lem:DZ-shuffle} and the result of \cite{Zha26}, while \cref{eq:depth-2-cusp} follows from the fact that $Z_{\mathcal{A}}(r)=0$ for even $r$. Therefore, for even $k>2$, we obtain the upper bound
$$\dim_{\mathbb{Q}}\mathcal{TZ}_k^{\mathcal{A}}
= \dim_{\mathbb{Q}}\mathcal{DZ}_k^{(\mathcal{A})}
\leq \left[ \frac{k+1}{2}\right] -2-\dim S_k(\mathrm{SL}_2(\mathbb{Z})).$$

\appendix
\section{The two definitions of $Z_{\mathcal{A}}(r,s)$}\label{sec:def-DFMZV}

In this section, we investigate the equivalence of the two definitions of
$Z_{\mathcal{A}}(r,s)$. Kaneko--Zagier \cite{KZ} define $\widetilde{Z}_{\mathcal{A}}(r,s)$ as follows and conjecture that $Z_{\mathcal{A}}(r,s)$ coincides with $\til{Z}_{\mathcal{A}}(r,s)$ when $r+s$ is even.

\begin{definition}
For integers $a,b>0$ with $k=a+b>2$, we define
\begin{align*}
\til{Z}_{\mathcal{A}}(a,b) &\ceq \til{\mathbb{B}}(a,b)
-\sum_{i=1}^{b-2}(-1)^i\binom{a+i}{a}Z_{\mathcal{A}}(a+i)Z_{\mathcal{A}}(b-i) - Z_{\mathcal{A}}(k)
\\
&+ (-1)^a\binom{k-1}{a}(H_{b-1}-H_{k-1})Z_{\mathcal{A}}(k-1)
+ (-1)^{a}\binom{k-1}{a}\left(\gamma_{\mathcal{A}} Z_{\mathcal{A}}(k-1)-Z_{\mathcal{A}}'(k-1)\right).
\end{align*}
Here, $H_n$ denotes the $n$-th harmonic number $1+1/2+1/3+\cdots+1/n$ (and set $H_0=0$). We also define
\begin{align*}
\til{\mathbb{B}}(a,b) \ceq \left(\sum_{i=a-1}^{p-1-b}\binom{i+1}{a}\frac{B_{i+1-a}}{i+1}\frac{B_{p-i-b}}{p-i-b}\mod p\right)_{p},
\end{align*}
\begin{align*}
\gamma_{\mathcal{A}} \ceq (W_p\mod p)_{p},
\quad \text{and}\quad
Z_{\mathcal{A}}'(k) \ceq \left(\frac{1}{p}\left(-\frac{B_{p-k}}{p-k}+\frac{B_{2p-1-k}}{2p-1-k}\right)\mod p\right)_{p},
\end{align*}
where $W_p$ denotes the Wilson quotient, defined by
\begin{align*}
W_p \ceq \frac{(p-1)!+1}{p}\in \Z_{>0}.
\end{align*}
Furthermore, we set $\widetilde{Z}_{\mathcal A}(1,1)\ceq 0$.
\end{definition}

The following statement implies that $Z_{\mathcal A}(r,s)=\widetilde Z_{\mathcal A}(r,s)$ for any integers $r,s>0$ such that $r+s$ is even. Indeed, combining it with the stuffle product relation $Z_{\mathcal A}(r)Z_{\mathcal A}(s) = Z_{\mathcal A}(r,s)+Z_{\mathcal A}(s,r)$, proved in \cite{Zha26}, immediately yields the desired equality. This fact is also noted in \cite{Zha26}.

\begin{thm}\label{thm:appendix}
For integers $r,s>0$ such that $r+s$ is even, we have
$$Z_{\mathcal{A}}(r,s)+\widetilde{Z}_{\mathcal{A}}(s,r) = Z_{\mathcal{A}}(r)Z_{\mathcal{A}}(s).$$
\end{thm}

In what follows, unless otherwise specified, \(\equiv\) denotes congruence modulo \(p\).

\begin{lem}
For integers $a\geq 0$, $b>0$, and $p>a+b$, we have
\begin{equation}\label{eq:lem-A1}
\binom{p-1-a}{b} \equiv (-1)^{b}\binom{a+b}{b} \mod p.
\end{equation}
Moreover, for integers $a,k>0$, $1\leq n\leq k-1$, and $p>k+a-1$, we have
\begin{equation}\label{eq:lem-A2}
\sum_{r=0}^{k-1}\binom{a+r-1}{a+n-1}\binom{p-k+r}{a+r} \equiv (-1)^{a+n} \mod p.
\end{equation}
\end{lem}

\begin{proof}
Equation \cref{eq:lem-A1} is easily obtained as follows.
\begin{align*}
\binom{p-1-a}{b} &= \frac{(p-1-a)(p-2-a)\cdots(p-b-a)}{b(b-1)\cdots 2\cdot 1}
\equiv (-1)^b\frac{(a+1)(a+2)\cdots(a+b)}{b(b-1)\cdots 2\cdot 1}
= (-1)^b\binom{a+b}{b}.
\end{align*}
We prove Equation \cref{eq:lem-A2}. Since the case $n=k-1$ follows from \cref{eq:lem-A1}, we may assume that $n<k-1$. This assumption is needed to justify the identity obtained from the beta function. Using the identity
\begin{equation}\label{eq:binom-fund-id-3}
\binom{a}{k}\binom{k}{b} = \binom{a}{b}\binom{a-b}{k-b}\ ,\quad (a,b,k\geq 0),
\end{equation}
we have
\begin{align*}
\sum_{r=0}^{k-1}\binom{a+r-1}{a+n-1}\binom{p-k+r}{a+r}
&\overset{\cref{eq:lem-A1}}{\equiv} \sum_{r=0}^{k-1}(-1)^{a+r}\binom{a+r-1}{a+n-1}\binom{k+a-1}{a+r}
\\
&= \sum_{r=0}^{k-1}(-1)^{a+r}\frac{a+n}{a+r}\binom{a+r}{a+n}\binom{k+a-1}{a+r}
\\
&\overset{\cref{eq:binom-fund-id-3}}{=} \sum_{r=n}^{k-1}(-1)^{a+r}\frac{a+n}{a+r}\binom{k+a-1}{a+n}\binom{k-n-1}{r-n}
\\
&= (a+n)\binom{k+a-1}{a+n}\sum_{r=0}^{k-1-n}(-1)^{a+r+n}\frac{1}{a+r+n}\binom{k-n-1}{r}
\end{align*}
Furthermore, by the well-known formula for the beta function, we have
\begin{align*}
\sum_{r=0}^{k-1-n}(-1)^{r}\frac{1}{a+r+n}\binom{k-n-1}{r}
&= \int_0^1 x^{a+n-1}(1-x)^{k-n-1}dx
\\
&= \frac{(a+n-1)!(k-n-1)!}{(k+a-1)!} = \frac{1}{a+n}\binom{k+a-1}{a+n}^{-1}.
\end{align*}
Combining these equations, we obtain the desired formula.
\end{proof}

\begin{lem}[Proposition 3 in \cite{KZ}]\label{lem:lem-A3}
For integers $r,s>0$ such that $r+s+1$ is even, we have
\begin{align*}
\zeta_{<p}(r,1,s) \equiv -\sum_{a=r-1}^{p-s-1} \frac{B_{a+1-r}}{a+1}\binom{a+1}{r}\frac{B_{p-s-a-1}}{p-s}\binom{p-s}{a+1} \mod p.
\end{align*}
\end{lem}

\begin{lem}\label{lem:lem-A4}
For an even integer $k>2$, an integer $1\leq n\leq k-1$, and a sufficiently large prime $p$, we have
\begin{align*}
\sum_{r=1}^{k-2}\binom{r}{n}\zeta_{<p}(r,1,k-1-r)
\equiv - &\sum_{a=1}^{p-k} \frac{B_{a}}{a}\frac{B_{p-k+1-a}}{p-k+1-a}
\left(\binom{p-a}{n}-\frac{1}{2}\binom{k-1}{n}\right)
\\
& -\frac{B_{p-k+1}}{p-k+1}\left(\frac{(-1)^{n}}{n} + (H_{k-1}-H_{k-n-1})\binom{k-1}{n}\right) \mod p ,
\end{align*}
\end{lem}

\begin{proof}
We first assume that $1\leq n\leq k-2$ in order to apply the beta-function identity below. By \cref{lem:lem-A3} and the identity
\begin{equation}\label{eq:binom-fund-id-1}
\frac{1}{a+1}\binom{a+1}{b+1} = \frac{1}{b+1}\binom{a}{b}\ ,\quad (a,b\geq 0),
\end{equation}
for integers $r,s>0$ with $k=r+s+1$ and a sufficiently large prime $p$, we have
\begin{align*}
\zeta_{<p}(r,1,s) &\equiv -\sum_{a=r-1}^{p-s-1} \frac{B_{a+1-r}}{a+1}\binom{a+1}{r}\frac{B_{p-s-a-1}}{p-s}\binom{p-s}{a+1}
\\
&= -\sum_{a=0}^{p-r-s} \frac{B_{a}}{a+r}\binom{a+r}{r}\frac{B_{p-r-s-a}}{p-s}\binom{p-s}{a+r}\quad (\because a\mapsto a+r-1)
\\
&\overset{\cref{eq:binom-fund-id-1}}{=}
-\sum_{a=1}^{p-k} \frac{B_{a}}{a}\binom{a+r-1}{a-1}\frac{B_{p-r-s-a}}{p-r-s-a}\binom{p-s-1}{p-k-a}
-\left(\frac{1}{r}+\frac{1}{p-s}\right)\frac{B_{p-r-s}}{p-r-s}\binom{p-s-1}{p-r-s-1}
\\
&\overset{\cref{eq:lem-A1}}{\equiv}
-\sum_{a=1}^{p-k} \frac{B_{a}}{a}\binom{a+r-1}{r}\frac{B_{p-k+1-a}}{p-k+1-a}\binom{p-k+r}{r+a}
-(-1)^{r}\left(\frac{1}{r}-\frac{1}{k-r-1}\right)\frac{B_{p-k+1}}{p-k+1}\binom{k-1}{r}.
\end{align*}
Here, we separate the last line of the above equation into two parts, corresponding to the latter and former terms. By using the identity
\begin{equation}\label{eq:binom-fund-id-2}
\binom{r}{n}\binom{a+r-1}{r} = \binom{a+n-1}{n}\binom{a+r-1}{a+n-1},
\end{equation}
we have
\begin{align*}
-\sum_{r=1}^{k-2}&\binom{r}{n}\sum_{a=1}^{p-k} \frac{B_{a}}{a}\binom{a+r-1}{r}\frac{B_{p-k+1-a}}{p-k+1-a}\binom{p-k+r}{r+a}
\\
&\overset{\cref{eq:binom-fund-id-2}}{=} -\sum_{a=1}^{p-k} \frac{B_{a}}{a}\binom{a+n-1}{n}\frac{B_{p-k+1-a}}{p-k+1-a}\sum_{r=0}^{k-2}\binom{a+r-1}{a+n-1}\binom{p-k+r}{r+a}
\\
&\overset{\cref{eq:lem-A2}}{\equiv} -\sum_{a=1}^{p-k} \frac{B_{a}}{a}\binom{a+n-1}{n}\frac{B_{p-k+1-a}}{p-k+1-a}\left((-1)^{a+n}-\binom{a+k-2}{a+n-1}\binom{p-1}{k-1+a}\right)
\\
&\overset{\cref{eq:lem-A1}}{\equiv} -\sum_{a=1}^{p-k} \frac{B_{a}}{a}\binom{a+n-1}{n}\frac{B_{p-k+1-a}}{p-k+1-a}\left((-1)^{a+n}+(-1)^{a+k}\binom{a+k-2}{a+n-1}\right)
\\
&\overset{\cref{eq:binom-fund-id-2}}{=} -\sum_{a=1}^{p-k} \frac{B_{a}}{a}\frac{B_{p-k+1-a}}{p-k+1-a}\left((-1)^{a+n}\binom{a+n-1}{n}+(-1)^{a+k}\binom{k-1}{n}\binom{a+k-2}{k-1}\right).
\\
&\overset{\cref{eq:lem-A1}}{\equiv} -\sum_{a=1}^{p-k} \frac{B_{a}}{a}\frac{B_{p-k+1-a}}{p-k+1-a}
\left((-1)^{a}\binom{p-a}{n}-(-1)^{a}\binom{k-1}{n}\binom{p-a}{k-1}\right).
\end{align*}
Since all terms with odd $a$ vanish, we may omit the factor $(-1)^a$.
On the other hand, we have
\begin{align*}
-\sum_{r=1}^{k-2}&\binom{r}{n}(-1)^{r}\left(\frac{1}{r}-\frac{1}{k-r-1}\right)\frac{B_{p-k+1}}{p-k+1}\binom{k-1}{r}
\\
&\overset{\cref{eq:binom-fund-id-3}}{=} -\frac{B_{p-k+1}}{p-k+1}\binom{k-1}{n}\sum_{r=n}^{k-2}(-1)^{r}\left(\frac{1}{r}-\frac{1}{k-r-1}\right)\binom{k-n-1}{r-n}
\\
&= -\frac{B_{p-k+1}}{p-k+1}\binom{k-1}{n}\sum_{r=0}^{k-2-n}(-1)^{r+n}\left(\frac{1}{r+n}-\frac{1}{k-n-r-1}\right)\binom{k-n-1}{r}
\quad (\because r\mapsto r+n).
\end{align*}
Here, we have
\begin{align*}
\sum_{r=0}^{k-2-n}(-1)^{r}\frac{1}{r+n}\binom{k-n-1}{r}
&= \int_{0}^{1}x^{n-1}(1-x)^{k-n-1}dx - (-1)^{k-n-1}\frac{1}{k-1}
\\
&= \frac{(n-1)!(k-n-1)!}{(k-1)!} - (-1)^{k-n-1}\frac{1}{k-1},
\end{align*}
where the last equality follows from the well-known formula for the beta function. Moreover, using the well-known formula
$$H_n = \sum_{i=1}^{n}(-1)^{i-1}\frac{1}{i}\binom{n}{i}$$
for harmonic numbers, we obtain
\begin{align*}
\sum_{r=0}^{k-2-n}(-1)^{r}\frac{1}{k-n-1-r}\binom{k-n-1}{r}
&= \sum_{r=1}^{k-n-1}(-1)^{k-n-1-r}\frac{1}{r}\binom{k-n-1}{r} \quad (\because r\mapsto k-n-1-r)
\\
&= (-1)^{k-n}H_{k-n-1}.
\end{align*}
Thus, we have
\begin{align*}
-\sum_{r=1}^{k-2}&\binom{r}{n}(-1)^{r}\left(\frac{1}{r}-\frac{1}{k-r-1}\right)\frac{B_{p-k+1}}{p-k+1}\binom{k-1}{r}
\\
&= -\frac{B_{p-k+1}}{p-k+1}\left(\frac{(-1)^{n}}{n} + \frac{(-1)^{k}}{k-1}\binom{k-1}{n} - (-1)^{k}\binom{k-1}{n}H_{k-n-1}\right)
\end{align*}
Combining the formulas obtained in the above discussion, for an even integer $k$, an integer $1\leq n\leq k-2$, and a sufficiently large prime $p$, we obtain
\begin{align*}
\sum_{r=1}^{k-2}\binom{r}{n}\zeta_{<p}(r,1,k-1-r)
\equiv - &\sum_{a=1}^{p-k} \frac{B_{a}}{a}\frac{B_{p-k+1-a}}{p-k+1-a}
\left(\binom{p-a}{n}-\binom{k-1}{n}\binom{p-a}{k-1}\right)
\\
& -\frac{B_{p-k+1}}{p-k+1}\left(\frac{(-1)^{n}}{n} + \frac{1}{k-1}\binom{k-1}{n} - \binom{k-1}{n}H_{k-n-1}\right) \mod p .
\end{align*}
Using Miki's identity \cite{Mik78}:
\begin{align*}
\frac{n}{2}\sum_{j=2}^{n-2}\frac{B_{j}}{j}\frac{B_{n-j}}{n-j} = \sum_{j=2}^{n-2}\binom{n}{j}B_{j}\frac{B_{n-j}}{n-j} + H_nB_n \ , \quad (n\geq 4)
\end{align*}
and the congruence
$$\frac{1}{a}\binom{p-a}{k-1} \equiv (-1)^{a+k-1}\frac{1}{k-1}\binom{p-k+1}{a} \mod p,$$
we have
\begin{align}
\sum_{a=1}^{p-k} \frac{B_{a}}{a}\frac{B_{p-k+1-a}}{p-k+1-a}\binom{p-a}{k-1}
&\equiv -\frac{1}{k-1}\sum_{a=1}^{p-k} B_{a}\frac{B_{p-k+1-a}}{p-k+1-a}\binom{p-k+1}{a} \notag
\\
&\equiv \frac{1}{2}\sum_{a=1}^{p-k} \frac{B_{a}}{a}\frac{B_{p-k+1-a}}{p-k+1-a}
- H_{p-k+1}\frac{B_{p-k+1}}{p-k+1}. \label{eq:from-miki-identity}
\end{align}
Hence, we have
\begin{align*}
\sum_{r=1}^{k-2}\binom{r}{n}\zeta_{<p}(r,1,k-1-r)
\equiv - &\sum_{a=1}^{p-k} \frac{B_{a}}{a}\frac{B_{p-k+1-a}}{p-k+1-a}
\left(\binom{p-a}{n}-\frac{1}{2}\binom{k-1}{n}\right)
\\
& -\frac{B_{p-k+1}}{p-k+1}\left(\frac{(-1)^{n}}{n} + (H_{k-1}-H_{k-n-1})\binom{k-1}{n}\right) \mod p ,
\end{align*}
where we use the congruence $H_{p-k}\equiv H_{k-1}\mod p$. This gives the desired formula for $1\leq n\leq k-2$. It remains to consider the case $n=k-1$. In this case, the left-hand side in the statement is zero. On the other hand, by \cref{eq:from-miki-identity}, we have
\begin{align*}
-\sum_{a=1}^{p-k}\frac{B_a}{a}\frac{B_{p-k+1-a}}{p-k+1-a}
&\left(\binom{p-a}{k-1}-\frac12\right)
-\frac{B_{p-k+1}}{p-k+1}\left(H_{k-1}-\frac{1}{k-1}\right)
\\
&\equiv \frac{B_{p-k+1}}{p-k+1}\left(H_{p-k+1}-H_{k-2}\right)
\equiv 0 \pmod p,
\end{align*}
where we used
$H_{p-k+1}\equiv H_{k-2}\mod p$.
Thus, the formula also holds for $n=k-1$.
\end{proof}

\begin{lem}\label{lem:lem-A5}
For an even integer $k>2$ and a sufficiently large prime $p$, we have
\begin{align*}
\frac{1}{2}\sum_{\substack{i=2\\i\neq p-k+1}}^{p-2}\frac{B_{i}}{i}\frac{B_{2p-k-i}}{2p-k-i}
&\equiv W_p\frac{B_{p-k+1}}{p-k+1}
+ \frac{1}{p}\left(\frac{B_{2p-k}}{2p-k}-\frac{B_{p-k+1}}{p-k+1}\right)\mod p . 
\end{align*}
\end{lem}

\begin{proof}
From the equations (32) and (42) in \cite{Lev20}, we have
\begin{align*}
\sum_{i=p-k+3}^{p-3}\frac{B_{i}}{i}\frac{B_{2p-k-i}}{2p-k-i}
&= - \sum_{i=2}^{p-k-1}\frac{B_{i}}{i}\frac{B_{p-k-i+1}}{p-k-i+1} + 2H_{p-k+1}\frac{B_{p-k+1}}{p-k+1}
\\
&+2\left(\binom{2p-k}{p-1}-1\right)\frac{B_{p-k+1}}{p-k+1}\frac{B_{p-1}}{p-1}
+2\frac{B_{2p-k}}{2p-k}H_{2p-k} - 4\frac{B_{p-k+1}}{p-k+1}H_{p-k+1}.
\end{align*}
Thus, for an even integer $k>2$ and a sufficiently large prime $p$, we have
\begin{equation}\label{eq:lem2-eq1}
\begin{split}
\frac{1}{2}\sum_{\substack{i=2\\i\neq p-k+1}}^{p-2}&\frac{B_{i}}{i}\frac{B_{2p-k-i}}{2p-k-i}
+ \frac{B_{p-k+1}}{p-k+1}H_{p-k+1}
\\
&\equiv \frac{B_{2p-k}}{2p-k}H_{2p-k}
+\left(\binom{2p-k}{p-1}-1\right)\frac{B_{p-k+1}}{p-k+1}\frac{B_{p-1}}{p-1}\mod p ,
\end{split}
\end{equation}
where we use Kummer's congruence. Here, using the identity
\begin{equation*}
\binom{p-1+m}{m} = \frac{p}{m}\prod_{j=1}^{m-1}\left(1+\frac{p}{j}\right) \equiv \frac{p}{m} \mod p^2
\end{equation*}
and $pB_{p-1}\equiv -1\mod p$ from the von Staudt--Clausen theorem, we have
\begin{equation}\label{eq:lem2-eq2}
\begin{split}
\binom{2p-k}{p-1}\frac{B_{p-k+1}}{p-k+1}\frac{B_{p-1}}{p-1}
&\equiv \frac{p}{p-k+1}\frac{B_{p-k+1}}{p-k+1}\frac{B_{p-1}}{p-1} \mod p^2
\\
&\equiv -\frac{1}{k-1}\frac{B_{p-k+1}}{p-k+1} \mod p .
\end{split}
\end{equation}
Using Glaisher's congruence \cite{Gla00}: $(p-1)! \equiv pB_{p-1}-p\mod p^2$, we have
$$\frac{pB_{p-1}}{p-1}\equiv\frac{pW_p}{p-1}+1\equiv 1-(p+1)pW_p\equiv 1-pW_p \mod p^2.$$
Using this formula, we have
\begin{align*}
p\frac{B_{2p-k}}{2p-k}&H_{2p-k} - p\frac{B_{p-k+1}}{p-k+1}\frac{B_{p-1}}{p-1}
\\
&\equiv p\frac{B_{2p-k}}{2p-k}H_{2p-k} - \frac{B_{p-k+1}}{p-k+1}(1-pW_p) \mod p^2
\\
&= p\frac{B_{2p-k}}{2p-k}\left(H_{p-1}+\frac{1}{p+1}+\cdots+\frac{1}{2p-k}\right) + \frac{B_{2p-k}}{2p-k} - \frac{B_{p-k+1}}{p-k+1} + pW_p\frac{B_{p-k+1}}{p-k+1}.
\end{align*}
Hence, we have
\begin{equation}\label{eq:lem2-eq3}
\begin{split}
&\frac{B_{2p-k}}{2p-k}H_{2p-k} - \frac{B_{p-k+1}}{p-k+1}\frac{B_{p-1}}{p-1}
\\
&\equiv \frac{B_{2p-k}}{2p-k}\left(H_{p-1}+\frac{1}{p+1}+\cdots+\frac{1}{2p-k}\right) + \frac{1}{p}\left(\frac{B_{2p-k}}{2p-k} - \frac{B_{p-k+1}}{p-k+1}\right) + W_p\frac{B_{p-k+1}}{p-k+1} \mod p.
\\
&\equiv (H_{k-1}+W_p)\frac{B_{p-k+1}}{p-k+1} + \frac{1}{p}\left(\frac{B_{2p-k}}{2p-k} - \frac{B_{p-k+1}}{p-k+1}\right) \mod p,
\end{split}
\end{equation}
where we use Kummer's congruence and $H_{p-k}\equiv H_{k-1}\mod p$. Therefore, by combining \cref{eq:lem2-eq1,eq:lem2-eq2,eq:lem2-eq3}, we have
\begin{align*}
\frac{1}{2}\sum_{\substack{i=2\\i\neq p-k+1}}^{p-2}\frac{B_{i}}{i}\frac{B_{2p-k-i}}{2p-k-i}
&\equiv W_p\frac{B_{p-k+1}}{p-k+1} + \frac{1}{p}\left(\frac{B_{2p-k}}{2p-k} - \frac{B_{p-k+1}}{p-k+1}\right) \mod p ,
\end{align*}
This gives the desired formula.
\end{proof}

\begin{proof}[\textbf{Proof of \cref{thm:appendix}}]
In what follows, we assume that $k=r+s>2$. This assumption is needed when applying \cref{lem:lem-A5}. When $k=2$, we have $r=s=1$, and the assertion follows immediately from
$$Z_{\mathcal A}(1,1)=\widetilde Z_{\mathcal A}(1,1)=0.$$
By the definition of $Z_{\mathcal{A}}(r,s)$ for even $k=r+s$, we have
\begin{align*}
Z_{\mathcal{A}}(r,s)
&= (-1)^{r}\sum_{b=1}^{k-2}
\left(\binom{b}{s}- \frac{1}{2} \binom{k-1}{s}\right)
\left(\zeta_{\mathcal{A}}(b,1,k-b-1)
-Z_{\mathcal{A}}(b)Z_{\mathcal{A}}(k-b)\right).
\end{align*}
Combining this with \cref{lem:depth3-0-formula}, the $p$-component of $Z_{\mathcal{A}}(r,s)$ as an element of $\mathcal{A}$, denoted by $Z_{\mathcal{A}}(r,s)_{(p)}$, can be expressed as
\begin{align*}
Z_{\mathcal{A}}(r,s)_{(p)}
&\equiv (-1)^{r}\sum_{b=1}^{k-2}\binom{b}{s}\zeta_{<p}(b,1,k-b-1)
- (-1)^{r}\sum_{b=2}^{k-2}\left(\binom{b}{s}-\frac{1}{2} \binom{k-1}{s}\right)\frac{B_{p-b}}{p-b}\frac{B_{p-k+b}}{p-k+b}.
\end{align*}
Using \cref{lem:lem-A4}, for a sufficiently large prime $p$, we have
\begin{align*}
Z_{\mathcal{A}}(r,s)_{(p)}
\equiv - (-1)^{s}&\sum_{a=1}^{p-k} \frac{B_{a}}{a}\frac{B_{p-k+1-a}}{p-k+1-a}
\left(\binom{p-a}{s}-\frac{1}{2}\binom{k-1}{s}\right)
\\
& - (-1)^{s}\frac{B_{p-k+1}}{p-k+1}\left(\frac{(-1)^{s}}{s} + (H_{k-1}-H_{k-s-1})\binom{k-1}{s}\right)
\\
&- (-1)^{s}\sum_{b=2}^{k-2}\left(\binom{b}{s}-\frac{1}{2} \binom{k-1}{s}\right)\frac{B_{p-b}}{p-b}\frac{B_{p-k+b}}{p-k+b}.
\end{align*}
Applying Kummer's congruence to the first line and replacing $b$ by $p-a$ in the last sum, we obtain
\begin{align*}
Z_{\mathcal{A}}(r,s)_{(p)}
\equiv - (-1)^{s}&\sum_{\substack{a=1\\a\neq p-k+1}}^{p-2} \frac{B_{a}}{a}\frac{B_{2p-k-a}}{2p-k-a}
\left(\binom{p-a}{s}-\frac{1}{2}\binom{k-1}{s}\right)
\\
& - (-1)^{s}\frac{B_{p-k+1}}{p-k+1}\left(\frac{(-1)^{s}}{s} + (H_{k-1}-H_{k-s-1})\binom{k-1}{s}\right).
\end{align*}
Then, applying \cref{lem:lem-A5}, we obtain
\begin{equation}\label{eq:explicit-Z}
\begin{split}
Z_{\mathcal{A}}(r,s)_{(p)}
\equiv - (-1)^{s}&\sum_{\substack{a=1\\a\neq p-k+1}}^{p-2} \frac{B_{a}}{a}\frac{B_{2p-k-a}}{2p-k-a}\binom{p-a}{s}
+ (-1)^s\binom{k-1}{s}\frac{1}{p}\left(\frac{B_{2p-k}}{2p-k}-\frac{B_{p-k+1}}{p-k+1}\right)
\\
& - (-1)^{s}\frac{B_{p-k+1}}{p-k+1}\left(\frac{(-1)^{s}}{s} + (H_{k-1}-H_{k-s-1}-W_p)\binom{k-1}{s}\right).
\end{split}
\end{equation}
On the other hand, we consider $\til{Z}_{\mathcal{A}}(r,s)$ with $k=r+s$. By definition, we have
\begin{align*}
\til{\mathbb{B}}(r,s)_{(p)} &\equiv \sum_{i=r-1}^{p-1-s}\binom{i+1}{r}\frac{B_{i+1-r}}{i+1}\frac{B_{p-i-s}}{p-i-s}
\\
&\overset{\cref{eq:binom-fund-id-1}}{=} \sum_{i=r}^{p-1-s}\binom{i}{r}\frac{B_{i+1-r}}{i+1-r}\frac{B_{p-i-s}}{p-i-s}
+ \frac{1}{r}\frac{B_{p-k+1}}{p-k+1}
\\
&= \sum_{i=1}^{p-k}\binom{i+r-1}{r}\frac{B_{i}}{i}\frac{B_{p-i-k+1}}{p-i-k+1}
+ \frac{1}{r}\frac{B_{p-k+1}}{p-k+1} \quad (\because i\mapsto i+r-1)
\\
&\overset{\cref{eq:lem-A1}}{\equiv}
(-1)^r\sum_{i=1}^{p-k}\binom{p-i}{r}\frac{B_{i}}{i}\frac{B_{2p-i-k}}{2p-i-k}
+ \frac{1}{r}\frac{B_{p-k+1}}{p-k+1},
\end{align*}
where we use Kummer's congruence to obtain the last line. Furthermore, by definition of $\til{Z}_{\mathcal{A}}(r,s)$ and a simple change of variables together with the above expression of $\til{\mathbb{B}}(r,s)$, we obtain
\begin{equation}\label{eq:explicit-Z2}
\begin{split}
\til{Z}_{\mathcal{A}}(r,s)_{(p)} &\equiv (-1)^{r}\sum_{\substack{a=1\\a\neq p-k+1}}^{p-r-1}\frac{B_{a}}{a}\frac{B_{2p-k-a}}{2p-k-a}\binom{p-a}{r}
- (-1)^{r}\binom{k-1}{r}\frac{1}{p}\left(\frac{B_{2p-k}}{2p-k}-\frac{B_{p-k+1}}{p-k+1}\right)
\\
&+ (-1)^{r}\frac{B_{p-k+1}}{p-k+1}\left(\frac{(-1)^{r}}{r} + (H_{k-1}-H_{k-r-1}-W_p)\binom{k-1}{r}\right).
\end{split}
\end{equation}
Therefore, comparing the equations \cref{eq:explicit-Z,eq:explicit-Z2}, we obtain
$$Z_{\mathcal{A}}(r,s) + \til{Z}_{\mathcal{A}}(s,r)
= - (-1)^{s}\sum_{\substack{a=p-s\\a\neq p-k+1}}^{p-2} \frac{B_{a}}{a}\frac{B_{2p-k-a}}{2p-k-a}\binom{p-a}{s}.$$
If $s=1$, the sum on the right-hand side is empty.
If $r=1$, the only possible nonzero term, corresponding to $a=p-s=p-k+1$, is excluded.
If $r,s>1$, only the term $a=p-s$ can contribute, and this term is equal to
$Z_{\mathcal{A}}(r)Z_{\mathcal{A}}(s)$.
This completes the proof.
\end{proof}

\bibliographystyle{myamsalpha}
\bibliography{ref}

\end{document}